\documentclass[12pt]{amsart}

\title{Hamiltonian Floer homology for compact convex symplectic manifolds}
\author{Sergei Lanzat}
\address{University of Haifa, Department of Mathematics, Mount Carmel, Haifa 31905, Israel}
\email{lanzat.sergei@gmail.com}
\date{\today}
\makeatletter

\usepackage{amsmath,amssymb,amsthm,amsfonts,mathrsfs,amscd,amsopn}
\usepackage{calligra}
\usepackage{dsfont}
\usepackage{latexsym}
\usepackage[pdftex]{graphicx}
\input xy
\xyoption{all}
\usepackage{bbm}

\renewcommand{\(}{\left(}
\renewcommand{\)}{\right)}

\newcommand{\FAT}[1]{\mbox{{$\mathbb{#1}$}}}
\newcommand{\Fat}[1]{\mbox{{$\scriptstyle\mathbb{#1}$}}}
\newcommand{\CL}[1]{\mbox{{$\mathcal{#1}$}}}
\newcommand{\Cl}[1]{\mbox{{$\scriptstyle\mathcal{#1}$}}}
\newcommand{\KL}[1]{\mbox{{$\mathscr{#1}$}}}
\newcommand{\Kl}[1]{\mbox{{$\scriptstyle\mathscr{#1}$}}}
\newcommand{\til}[1]{\widetilde{#1}}
\renewcommand{\hat}[1]{\widehat{#1}}
\newcommand{\hCL}[1]{\mbox{{$\mathcal{\hat{#1}}$}}}

\newcommand{\hn}[1]{\mbox{$\|{#1}\|_{L^{(1,\infty)}}$}}

\newcommand{\ZZ}{\FAT{Z}}

\newcommand{\FF}{\FAT{F}}
\newcommand{\DD}{\FAT{D}}

\newcommand{\KK}{\FAT{K}}
\newcommand{\PP}{\FAT{P}}

\newcommand{\RR}{\FAT{R}}

\newcommand{\CC}{\FAT{C}}

\newcommand{\kk}{\Fat{K}}
\newcommand{\ff}{\Fat{F}}
\newcommand{\nn}{\Fat{N}}

\newcommand{\dd}{\Fat{D}}

\newcommand{\rr}{\Fat{R}}

\newcommand{\IFF}{\Leftrightarrow}

\newcommand{\cl}[1]{\overline{#1}}
\newcommand{\minus}{\smallsetminus}

\newcommand{\ve}{\varepsilon}

\newcommand{\x}{\mathbbm{x}}
\newcommand{\y}{\mathbbm{y}}
\newcommand{\z}{\mathbbm{z}}
\newcommand{\PH}[1]{\CL{P}_{#1}}

\newcommand{\tPH}[1]{\til{\CL{P}_{#1}}}
\newcommand{\tph}[1]{\til{\Cl{P}_{#1}}}
\newcommand{\del}{\partial}

\newcommand{\tHam}{\widetilde{\hbox{Ham}}}

\newcommand{\be}{\begin{itemize}}
\newcommand{\ee}{\end{itemize}}

\newcommand{\beq}{\begin{equation}}
\newcommand{\eeq}{\end{equation}}

\newcommand{\beqn}{\begin{equation}\nonumber}

\newcommand{\bea}{\begin{equation}\begin{aligned}}
\newcommand{\eea}{\end{aligned}\end{equation}}

\newcommand{\bean}{\begin{equation}\nonumber\begin{aligned}}

\DeclareMathOperator{\Span}{Span}

\DeclareMathOperator{\ev}{ev}

\DeclareMathOperator{\id}{id}
\DeclareMathOperator{\Crit}{Crit}
\DeclareMathOperator{\ind}{ind}
\DeclareMathOperator{\Spec}{Spec}
\DeclareMathOperator{\Hom}{Hom}
\DeclareMathOperator{\Ham}{Ham}
\DeclareMathOperator{\Symp}{Symp}
\DeclareMathOperator{\Cal}{Cal}

\DeclareMathOperator{\Ker}{Ker}

\newtheorem{thm}{Theorem}[section]
\newtheorem{thm*}{Theorem}

\newtheorem{lem*}[thm*]{Lemma}
\newtheorem{prop}[thm]{Proposition}
\newtheorem{cor}[thm]{Corollary}

\newtheorem{defn}[thm]{Definition}
\newtheorem{thm-defn}[thm]{Theorem-Definition}

\newtheorem{rem}[thm]{Remark}

\begin{document}

\begin{abstract}
We construct absolute and relative versions of Hamiltonian Floer homology algebras for strongly semi-positive compact symplectic manifolds with convex boundary, where the ring structures are given by the appropriate versions of the pair-of-pants products.  We establish  the absolute and relative Piunikhin--Salamon--Schwarz  isomorphisms between these Floer homology algebras and the corresponding absolute  and relative  quantum homology algebras. As a result, the absolute and relative analogues  of the spectral invariants on the group of compactly supported Hamiltonian diffeomorphisms are defined.

\end{abstract}

\maketitle

\section{Introduction.}
In \cite{F-S}  U. Frauenfelder and  F.  Schlenk  defined the Floer homology for  weakly exact compact convex symplectic manifolds. The authors also established the Piunikhin--Salamon--Schwarz (PSS) isomorphism between the ring of Floer homology and the ring of Morse homology of such  manifolds. This in turn led to the construction of the  spectral invariants on the group of compactly supported Hamiltonian diffeomorphisms analogous to the spectral invariants constructed by M. Schwarz in \cite{Sch} and by Y.- G. Oh in  \cite{Oh1} for closed symplectic manifolds.  We extend the definitions and the constructions of  U. Frauenfelder and  F.  Schlenk to the case of strongly semi-positive compact convex symplectic manifolds. As a result, we get absolute and relative versions of Hamiltonian Floer homology algebras, where the ring structures are given by the appropriate versions of the pair-of-pants products. We establish  the absolute and relative Piunikhin--Salamon--Schwarz  isomorphisms between these Floer homology algebras and the corresponding absolute  and relative  quantum homology algebras. This makes it possible to define the absolute and relative analogues  of the spectral invariants on the group of compactly supported Hamiltonian diffeomorphisms. In \cite{Lanzat-QM} we use these spectral invariants   to construct (partial) quasi-morphisms on the universal cover $\widetilde{\mathrm{Ham}}_c(M, \omega)$ of the group of compactly supported Hamiltonian diffeomorphisms for a certain class of non-closed strongly semi-positive symplectic manifolds $(M,\omega)$. This leads to a construction of (partial) symplectic quasi-states on the space  of continuous  functions on $M$ that are constant near infinity.

Finally, let us mention that after the preprint of the paper had been published,  A. Ritter pointed to us about  his  construction of  the Floer cohomology for strongly semi-positive non-compact convex symplectic manifolds in \cite{Rit3}. In the same paper he also sketched the construction of  the quantum intersection product on the corresponding locally finite quantum homology group. This corresponds to our $\ast_3$ product in  \eqref{equation: quantum intersection products}.

\subsection{Setting}
We shall always work over the base field $\mathbb{F}$, which is either $\mathbb{Z}_2$ or $\mathbb{Q}$. Consider a $2n$-dimensional compact symplectic manifold $(M,\omega)$ with non-empty boundary $\del M$. Recall the following

\begin{defn}\label{Def:Convex symplectic manifolds}(cf. \cite{F-S}, \cite{McDuff}, \cite{MS3})
\mbox{}
 The boundary  $\del M$  is called convex if there exists a Liouville vector field $X$ (i.e. $\CL{L}_X\omega=d\iota_X\omega=\omega$), which is defined in the neighborhood of $\del M$ and which is everywhere transverse to $\del M$, pointing outwards; equivalently, there exists a $1$-form $\alpha$ on $\del M$ such that $d\alpha=\omega\mid_{\del M}$ and such that $\alpha\wedge(d\alpha)^{n-1}$ is a volume form inducing the boundary orientation of $\del M\subset M$. Therefore, $(\del M,\ker\alpha)$ is a contact manifold, and that is why a convex boundary is also called of a contact type. A compact symplectic manifold $(M,\omega)$ with non-empty boundary $\del M$ is convex if $\del M$ is convex.  A non-compact symplectic manifold $(M,\omega)$ is \textit{convex} if there exists an increasing sequence of compact convex submanifolds $M_i\subset M$ exhausting $M$, that is, $$M_1\subset M_2\subset\ldots\subset M_i\subset\ldots\subset M\;\;\;\text{and}\;\;\;\bigcup\limits_{i}M_i=M.$$
\end{defn}

Recall  that an \textit{almost complex structure}  on a smooth $2n$-dimensional manifold $M$ is a section $J$ of the bundle $\mathrm{End}\ TM$ such that $J^2(x)=-\mathds{1}_{T_xM}$ for every $x\in M$. An almost complex structure $J$ on $M$ is called compatible with $\omega$ ( or $\omega$-compatible) if $g_J:=\omega\circ(\mathds{1}\times J)$ defines a Riemannian metric on $M$.
Denote by $\CL{J}(M,\omega)$ the space of all $\omega$-compatible almost complex structures on $(M,\omega)$.

Given $(M,\omega)$ and  $J\in\CL{J}(M,\omega)$, then $(TM,J)$ becomes a complex vector bundle and, as such, its first Chern class $c_1(TM,J,\omega)\in H^2(M;\ZZ)$ is defined. Note that since the space $\CL{J}(M,\omega)$  is non-empty and contractible, (see \cite[ Proposition $4.1, (i)$]{MS}), the class $c_1(TM,J,\omega)$  does not depend on $J\in\CL{J}(M,\omega)$, and we shall denote it just by $c_1(TM,\omega)$.

Denote by $H_2^S(M)$ the image of the Hurewicz homomorphism $\pi_2(M)\to H_2(M,\mathbb{Z})$. The homomorphisms  $c_1: H_2^S(M)\to\mathbb{Z}$ and $\omega: H_2^S(M)\to\mathbb{R}$ are given by $c_1(A):=c_1(TM,\omega)(A)$ and $\omega(A)=[\omega](A)$ respectively.

\begin{defn}\label{defn: semi-positive manifolds}(cf. \cite{H-S}, \cite{McDuff})
\mbox{}
A  symplectic $2n$-manifold $(M,\omega)$ is called strongly semi-positive, if $\omega(A)\leq 0$ for any $A\in H^S_2(M)$ with $2-n\leq c_1(A)<0$.
\end{defn}

Next, we recall the definition of the quantum homology of compact convex symplectic manifolds, see \cite{Lanzat-Thesis},  \cite{Lanzat}. First of all, recall  the definition of the intersection products for a manifold with boundary.
\begin{defn}\label{definition: Intersection products}
Homomorphisms
\begin{equation}\nonumber\begin{aligned}\label{equation: intersection products in homology}
&\bullet_1:H_i(M;\mathbb{F})\otimes H_j(M;\mathbb{F})\to H_{i+j-2n}(M;\mathbb{F})\\
&\bullet_2:H_i(M;\mathbb{F})\otimes H_j(M,\partial M;\mathbb{F})\to H_{i+j-2n}(M;\mathbb{F})\\
&\bullet_3:H_i(M,\partial M;\mathbb{F})\otimes H_j(M,\partial M;\mathbb{F})\to H_{i+j-2n}(M,\partial M;\mathbb{F})
\end{aligned}\end{equation}
given by
\begin{equation}\nonumber\begin{aligned}
&a\bullet_1 b:=\mathrm{PLD}_2\left(\mathrm{PLD}_2^{-1}(b)\cup \mathrm{PLD}_2^{-1}(a)\right)\\
&a\bullet_2 b:=\mathrm{PLD}_2\left(\mathrm{PLD}_2^{-1}(b)\cup \mathrm{PLD}_1^{-1}(a)\right)\\
&a\bullet_3 b:=\mathrm{PLD}_1\left(\mathrm{PLD}_1^{-1}(b)\cup \mathrm{PLD}_1^{-1}(a)\right)
\end{aligned}\end{equation}
are called the intersection products in homology.\\ Here,
$
H^j(M;\mathbb{F})\overset{\mathrm{PLD}_1}{\longrightarrow} H_{2n -j}(M,\partial M;\mathbb{F}), H^j(M,\partial M;\mathbb{F})\overset{\mathrm{PLD}_2}{\longrightarrow} H_{2n -j}(M;\mathbb{F})
$
are the Poincar\'{e}-Lefschetz  isomorphisms given by $\mathrm{PLD}_i(\alpha):=\alpha\cap[M,\partial M]$, $i=1,2$, where $[M,\partial M]$ is the relative fundamental class, i.e. the positive generator of $H_{2n}(M,\partial M;\mathbb{F})\cong\mathbb{F}.$
\end{defn}

Now, consider the following Novikov ring $\Lambda$. Let \begin{equation}\label{equation: G group of periods}
G:=G(M,\omega):=\frac12\omega\left(H_2^S(M)\right)\subseteq\mathbb{R}
\end{equation}
be the subgroup of half-periods of the symplectic form $\omega$ on spherical homology classes.  Let $s$ be a formal variable. Define the field $\mathbb{K}_G$ of generalized Laurent series in  $s$ over $\mathbb{F}$ of the form
\begin{equation}\label{equation:  field KG of generalized Laurent series}
\mathbb{K}_G:=\left\{f(s)=\sum\limits_{\alpha\in G}z_{\alpha}s^{\alpha}, z_{\alpha}\in\mathbb{F}|\ \#\{\alpha>c|z_{\alpha}\neq 0\}<\infty,\ \forall c\in\mathbb{R} \right\}
\end{equation}

\begin{defn}\label{definition: Novikov ring Lambda}
Let $q$ be a formal variable. The \textit{Novikov ring $\Lambda:=\Lambda_G$} is the ring of polynomials in $q,q^{-1}$ with coefficients in the field $\mathbb{K}_G$, i.e.
\begin{equation}\label{equation:  Novikov ring}
\Lambda:=\Lambda_G:=\mathbb{K}_G[q,q^{-1}].
\end{equation}
We equip the ring $\Lambda_G$ with the structure of a graded ring by setting $\deg(s)=0$ and $\deg(q)=1$.  We shall denote by $\Lambda_k$ the set of  elements of $\Lambda$ of degree $k$. Note that  $\Lambda_0=\mathbb{K}_G$. The ring $\Lambda$ admits the following valuation. The valuation $\nu:\KK_G\to G\cup\{-\infty\}$  on the field $\KK_G$ is given  by
\beq
\begin{cases}
&\nu\(f(s)=\sum\limits_{\alpha\in G}z_{\alpha}s^{\alpha}\):=\max\{\alpha| z_{\alpha}\neq 0\},\; f(s)\not\equiv 0\\
&\nu(0)=-\infty.
\end{cases}
\eeq
Extend $\nu $ to $\Lambda$ by $\nu(\lambda):=\max\{\alpha|p_{\alpha}\neq 0\}$,  where $\lambda$ is uniquely represented by $$\lambda=\sum\limits_{\alpha\in G}p_{\alpha}s^{\alpha},\;\; p_{\alpha}\in\FF[q,q^{-1}].$$
\end{defn}

The absolute quantum homology  $QH_*(M;\Lambda)$ and the relative  quantum homology  $QH_*(M,\partial M;\Lambda)$ are defined as follows. As  modules, they are graded modules over $\Lambda$ defined by $QH_*(M;\Lambda):=H_*(M;\mathbb{F})\otimes_{\scriptstyle\mathbb{F}}\Lambda$ and $QH_*(M,\partial M;\Lambda):=H_*(M,\partial M;\mathbb{F})\otimes_{\scriptstyle\mathbb{F}}\Lambda$.
A grading on both modules is given by $\deg(a\otimes zs^{\alpha}q^m)=\deg(a)+m$.  Next, we define the quantum products $\ast_l, l=1,2,3$, which are  deformations of  the classical intersection products $\bullet_l, l=1,2,3$. Choose a homogeneous basis  $\{e_k\}_{k=1}^{d}$ of $H_*(M;\mathbb{F})$, such that $e_1=[pt]\in H_0(M;\mathbb{F})$. Let  $\{e^{\vee}_k\}_{k=1}^{d}$ be the dual homogeneous basis of $H_*(M,\partial M;\mathbb{F})$ defined by $\langle e_i, \mathrm{PLD}_1^{-1}(e^{\vee}_j) \rangle=\delta_{ij}$, where $\langle\cdot,\cdot\rangle$ is the Kronecker pairing.
Consider the group \begin{equation}\label{equation: Gamma group of spherical classes}
\Gamma:=\Gamma(M,\omega):=\frac{H_2^S(M)}{\ker( c_1)\cap \ker(\omega)}.
\end{equation}
Let $A\in H_2^S(M)$ and let $[A]\in\Gamma$ be the image of $A$ in $\Gamma$. Bilinear homomorphisms of $\Lambda$-modules
\begin{equation}\begin{aligned}\label{equation: quantum intersection products}
&\ast_1:QH_*(M;\Lambda)\times QH_*(M;\Lambda)\to QH_*(M;\Lambda) \\
&\ast_2:QH_*(M;\Lambda)\times QH_*(M,\partial M;\Lambda)\to QH_*(M;\Lambda) \\
&\ast_3:QH_*(M,\partial M;\Lambda)\times QH_*(M,\partial M;\Lambda)\to QH_*(M,\partial M;\Lambda)
\end{aligned}\end{equation}
are given as follows. Let $a\in H_i(M;\mathbb{F}), b\in H_j(M;\mathbb{F})$ and let $c\in H_i(M,\partial M;\mathbb{F}), d\in H_j(M,\partial M;\mathbb{F})$. Then
\begin{equation}\begin{aligned}\label{equation: formula of quantum intersection products}
& a\ast_1 b:=\sum_{[A]\in\Gamma}\left(\sum_{i=1}^d\sum_{A'\in[A]}GW_{A',2,3}(a,b,e^{\vee}_i)e_i\right)\otimes s^{-\omega(A)}q^{-2c_1(A)},
\\
&a\ast_2 d:=\sum_{[A]\in\Gamma}\left(\sum_{i=1}^d\sum_{A'\in[A]}GW_{A',1,3}(a,d,e^{\vee}_i)e_i\right)\otimes s^{-\omega(A)}q^{-2c_1(A)},
\\
&c\ast_3 d:=\sum_{[A]\in\Gamma}\left(\sum_{i=1}^d\sum_{A'\in[A]}GW_{A',1,3}(c,d,e_i)e^{\vee}_i\right)\otimes s^{-\omega(A)}q^{-2c_1(A)},
\end{aligned}\end{equation}
with $\deg(a\ast_1 b)=\deg(a\ast_2 d)=\deg(c\ast_3 d)=i+j-2n$. We extend these $\mathbb{F}$-bilinear homomorphisms on classical homologies to $\Lambda$-bilinear homomorphisms on quantum homologies by $\Lambda$-linearity. Here,
$$
GW_{A,p,m}: H_*(M;\mathbb{F})^{\times p}\times H_*(M,\partial M;\mathbb{F})^{\times (m-p)}\to\mathbb{F}
$$
stands for the genus zero Gromov-Witten invariant relative to the boundary, see \cite{Lanzat-Thesis}, \cite{Lanzat}.

Like in the closed case, we have different natural pairings. The $\KK_G$-valued pairings are given by
\bea\label{equation: Delta pairing on QH}
&\Delta_1:QH_k(M)\times QH_{2n-k}(M)\to\Lambda_0=\KK_G,\\
&\Delta_2:QH_k(M)\times QH_{2n-k}(M,\partial M)\to\Lambda_0=\KK_G,\\
&\Delta_l\(a, b\):=\imath(a\ast_l b),\ \text {for}\ l=1,2,
\eea
where the map $$\imath:QH_0(M)=\bigoplus_iH_i(M;\FF)\otimes_{\ff}\Lambda_{-i}\to\KK_G$$ sends $[pt]\otimes f_0(s)+\sum_{m=1}^{2n} a_m\otimes f_m(s)q^{-m}$ to $f_0(s)$.  The $\FF$-valued pairings are given by
\bea\label{equation: Pi pairing on QH}
&\Pi_1:QH_k(M)\times QH_{2n-k}(M)\to\FF,\\
&\Pi_2:QH_k(M)\times QH_{2n-k}(M,\partial M)\to\FF,\\
&\Pi_l=\jmath\circ\Delta_l,\ \text {for}\ l=1,2,
\eea
where the map $\jmath:\KK_G\to\FF$ sends $f(s)=\sum_{\alpha}z_{\alpha}s^{\alpha}\in \KK_G$ to $z_0$.  Moreover, the pairings $\Delta_2$ and $\Pi_2$ are non-degenerate. Since the quantum homology groups are finite-dimensional $\KK_G$-vector spaces in each degree, it follows that the paring $\Delta_2$ gives rise to Poincar\'{e}-Lefschetz duality over the field $\KK_G$.

\subsection{Structure of the paper}
In Section 2 we construct absolute and relative versions of Floer homology groups for strongly semi-positive compact symplectic manifolds with convex boundary and show the Poincar\'{e}-Lefschetz duality between them. These groups are equipped with  ring structures by means of the appropriate versions of the pair-of-pants products.  We establish  the absolute and relative Piunikhin--Salamon--Schwarz  isomorphisms between these Floer homology algebras and the corresponding absolute  and relative  quantum homology algebras mentioned above.

In Section 3 we define the absolute and relative analogues of the spectral invariants on the group of compactly supported Hamiltonian diffeomorphisms. We show that these invariants satisfy the standard properties analogously to the closed case.

\section{Floer homology}
From now on let  $(M, \omega)$ be a strongly semi-positive  compact convex $2n$-dimensional symplectic manifold. In the following Sections $2.1-2.2$ we recall important technical notations and  facts discussed in \cite{F-S}.

\subsection{Completion of a convex symplectic manifold}
 Let $X$ be  a Liouville  vector field (see Definition~\ref{Def:Convex symplectic manifolds}), which is defined in some neighborhood of $\partial M$ and which is everywhere transverse to $\partial M$, pointing outwards. Using $X$ we can symplectically identify a neighborhood of $\partial M$ with
$$
\( \partial M \times (-2\ve, 0], d \( e^r \alpha \) \)
$$
for some $\ve >0$, where $\alpha=\iota_X\omega$ is the Liouville $1$-form. In this identification we used coordinates $(x,r)$ on $\partial M \times (-2\ve, 0]$, and in these coordinates, $X(x,r) = \frac{\partial}{\partial r}$
on $\partial M \times (-2\ve, 0]$. We can thus view $M$ as a compact subset of the non-compact symplectic manifold $(\widehat{M},\widehat{\omega})$
defined as
\begin{eqnarray*}
\widehat{M}      &=& M \cup_{\partial M \times \{0\}} \partial M \times [0,\infty), \\
\widehat{\omega} &=&
   \left\{ \begin{array}{lll}
          \omega & \text{on} & M, \\
          d \( e^r \alpha \) & \text{on} & \partial M \times (-2\ve,\infty),
           \end{array}
   \right.
\end{eqnarray*}
and $X$ smoothly extends to $\partial M \times (-2\ve, \infty)$ by
$$
\widehat{X} (x,r) := \frac{\partial}{\partial r}, \quad\,
(x,r) \in \partial M \times (-2\ve, \infty) .
$$
For any $r\in\RR$ we denote the open ``tube"
$\partial M \times (r, \infty)$ by $P_r$:
$$
P_r:=\partial M \times (r, \infty).
$$
Let $\phi^t:=\phi^t_{\widehat{X}}$ be the flow of $\widehat{X}$.
Then $\phi^r(x,0) = (x,r)$ for $(x,r) \in P_{-2\ve}$.
Choose an $\widehat{\omega}$-compatible almost complex structure $\widehat{J}$ on $\widehat{M}$, such that
\begin{eqnarray}
&&\widehat{\omega} \( \widehat{X}(x),\widehat{J} (x) v \) = 0,  \hspace{25mm}  x \in \partial M, \,\, v \in T_x \partial M,  \label{J3} \\
&&\widehat{\omega} \( \widehat{X}(x), \widehat{J}(x) \widehat{X}(x) \) = 1,  \hspace{18mm}  x \in \partial M, \label{J2} \\
&&d_{(x,0)} \phi^r  \widehat{J}(x,0) = \widehat{J} (x,r) d_{(x,0)} \phi^r,   \hspace{7mm}  (x,r) \in P_{-2\ve}, \label{J1}
\end{eqnarray}
\begin{defn}\label{defn: parametrized admissible a.c.s}
For any smooth  manifold $B$ define the subset $\hCL{J}_B$ of the set of smooth sections $\Gamma \big( \widehat{M} \times B, \mathrm{End} \big(T \widehat{M} \big) \big)$ by
$$
\widehat{J} \in \hCL{J}_B\ \IFF\  \widehat{J}_b := \widehat{J} (\cdot,b)\ \text{is}\ \widehat{\omega}\text{-compatible and satisfies \eqref{J3}, \eqref{J2} and \eqref{J1}.}
$$
For any $r\geq-2\ve$ define  $\hCL{J}_{B,P_r}$ to be the set of all $\widehat{J}\in\hCL{J}_B$ that are independent of the $b$-variable on $\cl{P_r}$. And, at last,  we define the set
\beqn
\CL{J}_{B, P_r} := \left\{ J \in \Gamma \(M \times B ,\mathrm{End}(TM) \) \mid J=\widehat{J}|_{M \times B} \text{ for some } \widehat{J} \in \hCL{J}_{B, P_r}\right\}.
\eeq
\end{defn}
By \cite[ Remark $4.1.2$]{BPS} or \cite[ discussion on page $106$]{CFH},
the space $\hCL{J}_{B, P_r} $ is non-empty and connected. Since the restriction map $\hCL{J}_{B, P_r} \to\CL{J}_{B, P_r}$ is continuous, $\CL{J}_{B, P_r}$ is also non-empty and  connected.
\par Let  $e \in  C^\infty \( P_{-2\ve} \)$ be given by $e(x,r) := e^r$.
\begin{thm}(\cite[ Theorem $2.1$]{F-S}) \label{t:convex}
For $h \in  C^\infty (\RR)$ define $H \in  C^\infty \( P_{-2\ve}\)$ by
$$H(p) = h (e(p)), \quad\, p \in P_{-2\ve}.$$
Let $\Omega$ be a domain in $\CC$ and let $\widehat{J} \in \hCL{J}_{\Omega}$. If $u \in  C^\infty \( \Omega,P_{-2\ve} \)$ is a solution of Floer's equation
\begin{equation}  \label{e:floer}
\partial_s u(z) + \widehat{J} (u(z),z) \partial_t u(z) + \nabla H(u(z))=0,\quad \,
z=s+it \in \Omega,
\end{equation}
then
\begin{equation}  \label{viterbo}
\Delta ( e(u) ) = g_{\widehat{J}} (\partial_s u, \partial_s u) - h''(e(u))\cdot \partial_s ( e(u) ) \cdot e(u) .
\end{equation}
\end{thm}

\begin{thm}[The case of a time-dependent Hamiltonian](\cite[ Remark $2.2$]{F-S}) \label{sconvex}
Let $h \in  C^\infty(\mathbb{R}^2,\mathbb{R})$ and define $H \in  C^\infty \( P_{-2\ve} \times \RR \)$ by
$$
H(p,s) = h(e(p),s), \quad\, p \in P_{-2\ve}, \,\,  s \in \RR .
$$
Let $\Omega$ be a domain in $\CC$ and let $\widehat{J} \in \hCL{J}_{\Omega}$. If  $u \in  C^\infty \( \Omega, P_{-2\ve} \)$ is a solution of the time-dependent Floer equation
\begin{equation}\label{sfloer}
\partial_s u(z)+ \widehat{J}(u(z),z) \partial_t u(z) + \nabla H(u(z),s)=0, \quad\, z=s+it \in \Omega ,
\end{equation}
then
$$\Delta (e(u)) = g_{\widehat{J}} (\partial_s u, \partial_s u)  - \partial^2_1h(e(u),s)\cdot \partial_s e(u) \cdot e(u) -\partial_1\partial_2 h(e(u),s)\cdot e(u) .$$
\end{thm}
\begin{cor}[Maximum Principle](\cite[ Corollary $2.3$]{F-S})\label{maximum}
Assume that $u \in  C^\infty \( \Omega, P_{-2\ve} \)$ and that one of the following conditions holds.
\begin{itemize}
 \item[(i)]$u$ is a solution of Floer's equation~\eqref{e:floer};
 \item[(ii)] $u$ is a solution of the time-dependent Floer equation~\eqref{sfloer}
  and $\partial_1\partial_2 h \leq 0$.
\end{itemize}
If $e\circ u$ attains its maximum on $\Omega$, then $e \circ u$ is constant.
\end{cor}
\subsection{Admissible Hamiltonians on $M$}\label{subsection: admissible Hamiltonians on M}
Recall that given a smooth function (Hamiltonian) $H:\mathbb{S}^1\times M\to\RR$ one defines a time-dependent smooth vector field $X_{H_t}$, called the Hamiltonian vector field of $H$, by $$X_{H_t}: M\to TM,\;\;\;\omega(X_{H_t},\cdot)=-dH_t(\cdot)\ \text{for all}\ t\in\mathbb{S}^1,$$ where $H_t(\cdot):= H(t,\cdot)$. A flow generated by $X_{H_t}$ will be denoted by $\phi^t_H$.

Let $R$ be the Reeb vector field of the Liouville form $\alpha$ on $\partial M$. Recall that $R$ is uniquely defined by
\begin{equation}\label{reeb}
\omega_x ( v,R ) = 0
\,\text{ and }\,
\omega_x (X,R) = 1 , \quad \,
\, x \in \partial M, \,\, v \in T_x \partial M.
\end{equation}
Note that  $R = \hat{J} \,\hat{X} |_{\partial M}$ for any $\hat{J}\in\hCL{J}_{B, P_{-2\ve}}$. It follows that for any $h \in  C^\infty ( \RR )$  the Hamilton equation
$\dot{x} = X_{H_t} (x)$ of $H = h \circ e:  P_{-2\ve} \to \RR$ restricted to $\partial M$ has the form
\begin{equation}  \label{hamilton}
\dot{x}(t) \,=\, h'(1) \,R(x(t)) .
\end{equation}
Define the Reeb period $\kappa \in (0, \infty ]$ of $R$ by
\bea\label{equation: Reeb period kappa}
\kappa :=\inf_{c>0} \left\{\dot{x}(t) = c\, R(x(t)) \text{ has a non-constant  $1$-periodic solution} \right\} .
\eea
Define  two sets of smooth functions $\hCL{H}^{\pm}\subseteq C^\infty(\mathbb{S}^1 \times \widehat{M})$ by
\begin{equation}\label{equation: preadmissible hamiltonians}
\widehat{H}\in\hCL{H}^+\Leftrightarrow \exists h\in C^\infty( \RR )\ \text{such that}\
\begin{cases}
&0 \leq - h'(e^r) < \kappa\ \forall r\ge 0,\\
&\widehat{H}|_{\scriptstyle\mathbb{S}^1 \times \cl{P_0} }= h \circ e,
\end{cases}
\end{equation}
$$
\widehat{H}\in\hCL{H}^-\Longleftrightarrow-\widehat{H}\in\hCL{H}^+.
$$
Note the sign convention of $h'(e^r) $. By \eqref{hamilton}, \eqref{equation: Reeb period kappa} and \eqref{equation: preadmissible hamiltonians}\,, we have that for any $\widehat{H} \in \hCL{H}^{\pm}$ the restriction of the flow $\phi_{\widehat{H}}^t$  to $\cl{P_0}$  has no non-constant $1$-periodic solutions. Next, we define two sets of \emph{\textbf{admissible Hamiltonian functions on $M$}} by
\hspace{-10mm}\beq\label{equation: admissible Hamiltonians}
\CL{H}^{\pm} := \left\{ H \in  C^\infty (\mathbb{S}^1 \times M ) \mid  H = \widehat{H}|_{\scriptstyle\mathbb{S}^1 \times M} \text{ for some } \widehat{H} \in \hCL{H}^{\pm} \right\}.
\eeq
Note that  the space  $\CL{H}_c(M)$ of $C^\infty$-smooth functions on  $\mathbb{S}^1\times M$, whose support is compact and is contained in $\mathbb{S}^1\times(M\minus\partial M)$ is a subset of $\CL{H}^+\cap\CL{H}^-$.

Given $H\in C^{\infty}(\mathbb{S}^1\times M)$, we denote the set of contractible $1$-periodic orbits of $\phi_H^t$ by $\PH{H}$, i.e.
$$\PH{H}:=\{x:\mathbb{S}^1\to M\;|\;x(t)=\phi^t_H(x(0)),\; x\;\text{is a contractible loop}\}.$$  For a generic (in the Baire sense in the Floer topology) $H \in \CL{H}^{\pm}$ we have
\begin{equation}  \label{e:floer:det}
\det \( \id - d_{x(0)} \phi^1_{H}\) \neq 0
\end{equation}
for all $x \in \PH{H}$. See \cite[Theorem 3.1]{H-S}. Since $M$ is compact, $\PH{H}$ is a finite set. An admissible $H$ satisfying \eqref{e:floer:det} for all $x \in \PH{H}$ is called {\it \textbf{regular}}, and the set of regular admissible Hamiltonians is denoted by $\CL{H}^{\pm}_{reg} \subset \CL{H}^{\pm}$. Note that for $H\in\CL{H}^{\pm}_{reg}$ the corresponding function $h\in C^\infty( \RR )$ satisfies $h'(1)\neq 0$.

Denote by $\CL{L}$ the space of smooth contractible loops $x: \mathbb{S}^1\to M$. Let  $\til{\CL{L}}$  be a covering of  $\CL{L}$, whose  elements  are equivalence classes $\x:=[x,\bar{x}]$ of pairs $(x,\bar{x})$, where $x \in \CL{L}$ and $\bar{x}: \DD^2 = \{z \in \CC \mid |z| \leq 1\} \to M$
satisfies $\bar{x}(e^{it})=x(t)$, and where $(x_1,\bar{x}_1)$ and $(x_2,\bar{x}_2)$ are equivalent if and only if
$$x_1 = x_2, \quad  \omega(\bar{x}_1 \#(- \bar{x}_2)) =0,\quad c_1(\bar{x}_1 \#(- \bar{x}_2)) =0.$$
The group $\Gamma$ acts on  equivalence classes $\x$ by
$A\x= [x,\bar{x}\#(-A)]$, for any $A \in \Gamma$, and $\CL{L} = \til{\CL{L}} / \Gamma$. Denote by $\tPH{H}$ the full lift of $\PH{H}$ to $\til{\CL{L}}$, i.e.
$\tPH{H}=\{\x=[x,\bar{x}]\in\til{\CL{L}}\,|\, x\in\PH{H} \}$. There exists an integral grading $\mu:\tPH{H}\to\ZZ$ by means of the Conley-Zehnder index, namely $\mu(\x)=n-\mu_{CZ}(x)$, see \cite[ Section $5$]{H-S}.  For a $C^2$-small autonomous Morse Hamiltonian $f$ and $x\in\Crit(f)$, we have that $\mu(\x)=\ind_f(x)$, where the critical point $x$ of $f$ is viewed as a constant path and $\bar{x}$ is the trivial disk. Under the action of $\Gamma$ the Conley-Zehnder grading behaves as follows: $\mu\(A\x\)=\mu(\x)+2c_1(A)$. For $H \in \CL{H}^{\pm}$ the \emph{\textbf{symplectic action functional}} $\CL{A}_H: \widetilde{\CL{L}}\to\RR$ is given by
\beq\label{equation: action functional}
\CL{A}_H(\x) \,:=\, -\int_{\dd^2} \bar{x}^*\omega+\int_{\scriptstyle\mathbb{S}^1}H(t, x(t)) \,dt.
\eeq
It is well-defined on $\widetilde{\CL{L}}$. Note that $\CL{A}_H(A\x)=\CL{A}_H(\x)+\omega(A)$ for any $A\in\Gamma$. We also note that the set  $\tPH{H}$ can be viewed as the set of critical points of  $\CL{A}_H$, see \cite[ Section $5.3$]{ABKLR}. We denote by $\Spec(H)$ the \emph{\textbf{action spectrum}} of $H$, i.e.
\beq\label{equation: action spectrum}
\Spec(H)=\left\{\CL{A}_H(\x)\in\RR\,|\, \x\in\tPH{H}\right\}.
\eeq
The action spectrum is a measure zero set of $\RR$, see \cite[ Lemma $2.2$]{Oh3}.

\subsection{Floer homology groups}
For a regular admissible  $H\in\CL{H}^{\pm}_{reg}$ , consider the free $\Lambda$-module $\tPH{H}\otimes\Lambda$. The grading on $\Lambda$ and the $\mu$-grading on $\tPH{H}$ give rise to the grading
\beq\label{equation:grading on Floer complex}
\deg:\tPH{H}\otimes\Lambda\to\ZZ,\;\;\;\deg(\x\otimes zs^{\alpha}q^m)=\mu(\x)+m.
\eeq
Denote by $R$ the $\Lambda$-submodule of $\tPH{H}\otimes\Lambda$ generated by $A\x\otimes 1 - \x\otimes s^{\omega(A)}q^{2c_1(A)}$, for all $ \x\in\tPH{H} , A\in\Gamma$. Since $\deg(A\x\otimes 1)=\deg( \x\otimes s^{\omega(A)}q^{2c_1(A)})=\mu\(A\x\)$,  we get the graded $\Lambda$-module $CF_*(H;\Lambda):=\cfrac{\tPH{H}\otimes\Lambda}{R}$.
The next step is to define a $\Lambda$-linear  differential $\partial:CF_*(H;\Lambda)\to CF_*(H;\Lambda)$ of a graded degree $-1$. This is a Morse-type differential that counts the algebraic number of isolated Floer cylinders connecting critical points of $\CL{A}_H$, i.e. elements of  $\tPH{H}$. Take $J\in\CL{J}_{\scriptstyle\mathbb{S}^1, P_{-2\ve}}$ and for each pair $\x=[x,\bar{x}]$ and $\y=[y,\bar{y}]$ in $\tPH{H}$, let $\CL{M}(\x,\y;H,J)$ be the moduli space of Floer connecting cylinders from $\x$ to $\y$, namely the set of solutions $u\in C^{\infty}(\RR\times\mathbb{S}^1,M)$ of the problem
\beq\label{equation: Floer connecting moduli space}
\left\{
\begin{aligned}
&\partial_s u + J_t(u)\(\partial_tu-X_{H_t}(u)\)=0,\\
&\lim\limits_{s\to-\infty}u(s,t)=x(t),\;\;\lim\limits_{s\to+\infty}u(s,t)=y(t),\\
&\bar{x}\#u\#(-\bar{y})\ \text{represents the zero class in}\ \Gamma.
\end{aligned}
\right\}
\eeq
The additive group $\RR$ acts on $\CL{M}(\x,\y;H,J)$ by translations
$$\RR\ni\tau:u(s,t)\mapsto u(s+\tau,t).$$ Denote by  $\hCL{M}(\x,\y;H,J)$ the quotient space $\CL{M}(\x,\y;H,J)/\RR.$
By the ``Maximum Principle", see  Corollary~\ref{maximum}, every solution of ~\eqref{equation: Floer connecting moduli space} lies in $M\minus\del M$. So we can apply the results of \cite[ Section $3$]{H-S} and \cite[ Section $3$]{Seidel}. It follows that  for generic $J\in\CL{J}_{\scriptstyle\mathbb{S}^1, P_{-2\ve}}$  the space  $\CL{M}(\x,\y;H,J)$ is a smooth manifold of dimension $\mu(\x)-\mu(\y)$. Following \cite{F-S}, such a generic $J$ will be called \emph{\textbf{$H$-regular}}. A pair $(H,J)$ will be called regular, if  $H\in\CL{H}^{\pm}_{reg}$ and $J$ is $H$-regular. The manifold $\CL{M}(\x,\y;H,J)$ carries an orientation by orienting the determinant bundle of the Cauchy-Riemann operator, see \cite[Section $6$]{FH}.  Moreover, if $\mu(\x)-\mu(\y)=1$, then  the quotient space $\hCL{M}(\x,\y;H,J)$ is a compact zero-dimensional oriented manifold. Hence it is a finite set of points,  each of which is equipped with an orientation sign $\pm 1$. Denote by $n_{H,J}(\x, \y):=\#_{\ff} \hCL{M}(\x,\y;H,J)$ its algebraic (over $\FF$) number of elements, and define the differential by
\beq\label{equation: Floer differential}
\partial(\x\otimes 1):=\sum\limits_{\substack{\y\in\tph{H}\\ \mu(\y)=\mu(\x)-1 }} n_{H,J}(\x,\y)\y\otimes 1.
\eeq
It follows from \cite[ Theorem $3.3$]{H-S} and \cite[ Corollary $3.4$]{Seidel}  that  the sum in ~\eqref{equation: Floer differential}  is a finite linear combination  of elements of $\tPH{H}$ over $\Lambda$ modulo the relation submodule $R$. Hence, $\partial(\x\otimes 1)(\mathrm{mod}\ R)\in CF_*(H;\Lambda)$ for all $\x\in\tPH{H}$.  Extending $\partial$ by $\Lambda$-linearity to the whole $CF_*(H;\Lambda)$, we get a well-defined $\Lambda$-homomorphism $$\partial:=\partial_{H,J}:CF_*(H;\Lambda)\to CF_{*-1}(H;\Lambda).$$ Repeating the original proof of Floer in \cite{F3},  we conclude that $\partial^2=0$. It follows that $(CF_*(H;\Lambda), \partial_{H,J})$ is a chain complex and its homology  $H_*(CF_*(H;\Lambda), \partial_{H,J})$ is called the Floer homology over $\Lambda$ and will be denoted by $HF_*(H,J;\Lambda)$.

Now, suppose $(H_\alpha,J_\alpha)$ and $(H_\beta,J_\beta)$ are two regular pairs. We would like to compare the corresponding $HF_*(H_\alpha,J_\alpha;\Lambda)$ and $HF_*(H_\beta,J_\beta;\Lambda)$. Let $H_s$ be a homotopy connecting $H_\alpha$ and $H_\beta$. Recall that on the tube $P_0$, the homotopy has a form $H_s=h_s(e^r)$. If $\partial_s h'_s\leq 0$ on the tube $P_0$, the continuation map is well-defined, see \cite[Section $2.9$]{Rit1}. Note that if $\partial_s h'_s=0$ on the tube $P_0$, the continuation map is an isomorphism. But for general (monotone decreasing on $P_0$) homotopy, it is only a homomorphism.
\begin{rem}
For any $k\in\ZZ$ the sets $CF_k(H;\Lambda)$ and $HF_k(H,J;\Lambda)$ are finite-dimensional vector spaces over the field $\Lambda_0=\KK_G.$ There exists a basis of $CF_*(H;\Lambda)$ consisting of elements of the form $\x\otimes q^m$, with $\x\in\tPH{H}$.
\end{rem}
The action functional $\CL{A}_H$ and the valuation $\nu$  define the following filtration on the above $\Lambda$-modules. Let $(H,J)$ be a generic pair.  Define a map
$$\ell_H:CF_*(H;\Lambda)\to\RR\cup\{-\infty\}$$
by
\beq\label{equation: filtration ell}
\begin{cases}
&\ell_H\(c=\sum\limits_{i}\x_i\otimes\lambda_i\):=max\{\CL{A}_H(\x_i)+\nu(\lambda_i)|\ \lambda_i\neq 0\},\\
&\ell_H(0)=-\infty.
\end{cases}
\eeq
Since $\ell_H(A\x\otimes 1)=\ell_H\( \x\otimes s^{\omega(A)}q^{2c_1(A)}\)=\CL{A}_H(A\x)$,  the map $\ell_H$ is well-defined.
Let $\alpha\in\RR\minus\Spec(H)$ and define a  subspace $CF^{(-\infty,\alpha)}_*(H;\Lambda)$ of $CF_*(H;\Lambda)$ by
\bea\label{equation: filtered Floer chain complex}
&CF^{(-\infty,\alpha)}_*(H;\Lambda) := \{c\in CF_*(H;\Lambda)|\ell_H(c)<\alpha\}.
\eea
For any $u\in\CL{M}(\x,\y;H,J)$, the flow energy $$E(u):=\int\limits_{\rr\times\scriptstyle\mathbb{S}^1}g_{J}(\partial_su,\partial_su)ds\wedge dt$$  equals to $\CL{A}_H(\x)-\CL{A}_H(\y)\geq0.$  Hence
\beq\label{equation: differential preserves filtration}
\ell_H(\partial c)<\ell_H(c)
\eeq
for any $c\in CF_*(H;\Lambda)$. We conclude  that the differential $\partial$ preserves this subspace. Note that if $\omega\neq 0$, the subspace $CF^{(-\infty,\alpha)}_*(H;\Lambda)$ is only an  $\FF$-submodule of  $CF_*(H;\Lambda)$. Therefore,  the quotient space
\bea\label{equation: quotient filtered Floer chain complex}
&CF^{(\alpha,\infty)}_*(H;\Lambda) :=  \frac{CF_*(H;\Lambda)}{CF^{(-\infty,\alpha)}_*(H;\Lambda)}
\eea
is well-defined, and the differential $\partial$ induces a chain complex structure on it. Define the filtered Floer homology groups by
\bea\label{equation:  filtered Floer homology}
&HF^{(-\infty,\alpha)}_*(H,J;\Lambda) := H_*(CF^{(-\infty,\alpha)}_*(H,\Lambda),\partial),\\
&HF^{(\alpha,\infty)}_*(H,J;\Lambda) := H_*(CF^{(\alpha,\infty)}_*(H,\Lambda),\partial).
\eea

Let  $\pi^\alpha:CF_*(H;\Lambda)\twoheadrightarrow CF^{(\alpha,\infty)}_*(H;\Lambda)$ be the natural projection, and let $i^\alpha:CF^{(-\infty,\alpha)}_*(H;\Lambda)\hookrightarrow CF_*(H;\Lambda)$ be the inclusion map.
Then we have induced homomorphisms in homology
\bea\label{equation: induced homomorphisms in filtered Floer homology}
&i^\alpha_*:HF^{(-\infty,\alpha)}_*(H,J;\Lambda)\to HF_*(H,J;\Lambda),\\
&\pi^\alpha_*:HF_*(H,J;\Lambda)\to HF^{(\alpha,\infty)}_*(H,J;\Lambda).
\eea
The homological exact sequence yields $\Ker\(\pi^\alpha_*\)=\mathrm{Im}\(i^\alpha_*\)$.

\subsection{Poincar\'{e}-Lefschetz duality}
For any $H\in\CL{H}^{\pm}$ let $H^{(-1)}:\mathbb{S}^1\times M\to\RR$  be  the Hamiltonian function defined by $H^{(-1)}(t,x):=-H(-t,x)$. It follows that the flow $\phi^t_{H^{(-1)}}$ generated by $H^{(-1)}$ is given by $\phi^{1-t}_H\circ(\phi^1_H)^{-1}$  and that $H^{(-1)}\in\CL{H}^{\mp}$. Moreover, we have  $H\in\CL{H}^{\pm}_{reg}\ \IFF\  H^{(-1)}\in\CL{H}^{\mp}_{reg}$. In addition, there are  bijective correspondences
$$
\CL{P}_H\ni x\; \overset{1:1}{\longleftrightarrow}\; x^{(-1)}\in\CL{P}_{H^{(-1)}},
$$
$$
\tPH{H}\ni\x=[x,\bar{x}]\; \overset{1:1}{\longleftrightarrow}\; \x^{(-1)}=\left[x^{(-1)},\bar{x}^{(-1)}\right]\in\til{\CL{P}_{H^{(-1)}}},
$$
where $x^{(-1)}(t):=x(-t),\; \bar{x}^{(-1)}(z)=\bar{x}(\bar{z})$,
and
$$
\CL{M}(\x,\y;H,J)\ni u\;\overset{1:1}{\longleftrightarrow}\; u^{(-1)}\in\CL{M}\(\y^{(-1)},\x^{(-1)};H^{(-1)},J^{(-1)}\)
$$
where $u^{(-1)}(s,t):=u(-s,-t)$ and  $J^{(-1)}(t,p):=J(-t,p)$.
In view of these correspondences, we have
$\mu\(\x^{(-1)}\)=2n-\mu(\x)$ and $\CL{A}_H(\x)=-\CL{A}_{H^{(-1)}}\(\x^{(-1)}\)$. Following the general algebraic theory of  M. Usher, see \cite{Usher2},  we conclude that the chain complexes $CF_*(H,\Lambda)$ and $CF_*(H^{(-1)},\Lambda)$ are graded filtered (Floer-Novikov) complexes, which are opposite to each other. Therefore, following \cite{E-P1}, \cite{Ostr-qmm} and \cite{Usher2}, we define several pairings between opposite complexes. Firstly, let us fix  $\Lambda$-generators of  $CF_*(H,\Lambda)$, namely $ CF_*(H,\Lambda)=\Span_{\Lambda}(\x_1,\dots, \x_s),\; s=|\CL{P}_H|$. Here, we abuse notation while writing $\x$ instead of $\x\otimes 1$.
Then, of course, $CF_*(H^{(-1)},\Lambda)=\Span_{\Lambda}\(\x^{(-1)}_1,\dots, \x^{(-1)}_s\)$, and we define an $\Lambda$-valued pairing
$\KL{L}^\#:CF_*(H,\Lambda)\times CF_*\(H^{(-1)},\Lambda\)\to\Lambda$ by extending  $\Lambda$-linearly the relation $\KL{L}^\#\(\x_i,\x^{(-1)}_j\)=\delta_{ij}$, i.e.
\beq\label{equation: KLL pairing on Floer homology}
\KL{L}^\#\(\sum_{i=1}^s\lambda_i\x_i, \sum_{j=1}^s\lambda^\prime_j\x^{(-1)}_j\)=\sum_{ij}\lambda_i\lambda^\prime_j\delta_{ij}.
\eeq
The pairing  $\KL{L}^\#$ is non-degenerate and satisfies $\KL{L}^\#\(\partial\x,\y^{(-1)}\)= \KL{L}^\#\(\x,\delta\y^{(-1)}\)$,  for any $\x\in\tPH{H},\ \y\in\til{\CL{P}_{H^{(-1)}}}$. Here $\delta:=\partial_{H^{(-1)}, J^{(-1)}}$, see \cite[ Section $1.4$]{Usher2}.  For any $k\in\ZZ$ the pairing $\KL{L}^\#$ restricts to a non-degenerate $\KK_G$-valued pairing
$$
\CL{L}^\#:=\KL{L}^\#|_{CF_k(H,\Lambda)\times CF_{2n-k}\(H^{(-1)},\Lambda\)} : CF_k(H,\Lambda)\times CF_{2n-k}\(H^{(-1)},\Lambda\)\to\KK_G.
$$
Since $CF_k(H,\Lambda)$ and  $CF_{2n-k}\(H^{(-1)},\Lambda\)$ are finite-dimensional vector spaces over the field $\KK_G$, the pairing $\CL{L}^\#$ gives rise to an isomorphism
$$
CF_k(H,\Lambda)\cong\Hom_{\kk_G}\(CF_{2n-k}\(H^{(-1)},\Lambda\),\KK_G \).
$$
By the universal coefficient theorem, we obtain the Poincar\'{e}-Lefschetz duality isomorphism
\beq\label{equation: Poincare-Lefschetz duality for Floer homology  }
HF_k(H,J;\Lambda)\cong\Hom_{\kk_G}\(HF_{2n-k}\(H^{(-1)}, J^{(-1)};\Lambda\),\KK_G \).
\eeq
The equality  $\CL{L}^\#\(\partial\x,\y^{(-1)}\)=\CL{L}^\#\(\x,\delta\y^{(-1)}\)$ implies that $\CL{L}^\#$ descends to a pairing
\beq\label{equation: CLL pairing on Floer homology}
\CL{L}:HF_k(H,J;\Lambda)\times HF_{2n-k}\(H^{(-1)}, J^{(-1)};\Lambda\)\to\KK_G,
\eeq
which is non-degenerate by the Poincar\'{e}-Lefschetz duality isomorphism.

We can go further and consider a $\FF$-valued pairing
$$
L^\# : CF_k(H,\Lambda)\times CF_{2n-k}\(H^{(-1)},\Lambda\)\to\FF
$$
defined by $L^\#:=\jmath\circ\CL{L}^\#$. Recall that  the map $\jmath:\KK_G\to\FF$ sends $\sum_{\alpha}z_{\alpha}s^{\alpha}$ to $z_0$. The pairing $L^\#$ is non-degenerate and satisfies the equation
$$
L^\#\(\partial\x,\y^{(-1)}\)=n_{H,J}(\x,\y)=n_{H^{(-1)},J^{(-1)}}(\y^{(-1)},\x^{(-1)})=L^\#\(\x,\delta\y^{(-1)}\)
$$
for any $\x\in\tPH{H},\ \y^{(-1)}\in\til{\CL{P}_{H^{(-1)}}}$. This implies that $L^\#$ descends to a $\FF$-bilinear pairing
\beq\label{equation: L pairing on Floer homology}
L:=\jmath\circ\CL{L}:HF_k(H,J;\Lambda)\times HF_{2n-k}\(H^{(-1)}, J^{(-1)};\Lambda\)\to\FF.
\eeq
The paring $L$ is obviously non-degenerate. We can consider restrictions of $L^\#$ and $L$ to the filtered Floer chain groups, which are only $\FF$-vector subspaces of the full Floer chain groups. Firstly, we note that since $\mu\(\x^{(-1)}\)=2n-\mu(\x)$ and $\CL{A}_H(\x)=-\CL{A}_{H^{(-1)}}\(\x^{(-1)}\)$ we have $L^\#|_{CF^{(-\infty,\alpha]}_k(H,\Lambda)\times CF^{(-\infty,-\alpha)}_{2n-k}\(H^{(-1)},\Lambda\) }=0$. Hence, $L^\#$ descends to the quotient $CF^{(\alpha, \infty)}_k(H,\Lambda)$, namely we obtain a well-defined $\FF$-bilinear non-degenerate pairing
$$
L^\#_{\alpha}: CF^{(\alpha, \infty)}_k(H,\Lambda)\times  CF^{(-\infty,-\alpha)}_{2n-k}\(H^{(-1)},\Lambda\)\to\FF.
$$
Once again the equation $L^\#\(\partial\x,\y^{(-1)}\)=L^\#\(\x,\delta\y^{(-1)}\)$ leads to a well-defined $\FF$-bilinear pairing in homology
\beq\label{equation: filtered L pairing on Floer homology}
L_{\alpha}:HF^{(\alpha, \infty)}_k(H,J;\Lambda)\times HF^{(-\infty,-\alpha)}_{2n-k}\(H^{(-1)}, J^{(-1)};\Lambda\)\to\FF.
\eeq
A non-trivial theorem due to M. Usher \cite[ Theorem $1.3$]{Usher2} states that $L_{\alpha}$ is non-degenerate. Note that in the case $G=\omega(\Gamma)\leq\RR$ is a discrete subgroup, this theorem was proved by M.\ Entov and L.\ Polterovich in \cite{E-P1} and by Y.\ Ostrover in \cite{Ostr-qmm}.

\subsection{The pair-of-pants products}
In order  to define  ring structures on the Floer homology groups in terms of pair-of-pants products we shall follow the method of A. Ritter developed in \cite[Section $16$]{Rit2}.

Consider the following data:

\begin{itemize}
\item[$(i)$]a Riemann sphere $(S, j)$ with two negative and one positive punctures, with a fixed choice of complex structure $j$ and a fixed choice of parametrization $z = s + it\in(-\infty, 0]\times\mathbb{S}^1$ and $z = s + it\in[0,\infty)\times\mathbb{S}^1$ respectively near the negative and positive punctures so that $j\partial_s = \partial_t$. These parametrizations will be called the cylindrical ends.
\item[$(ii)$] a closed $1$-form $\beta$ on $S$, such that  on the negative cylindrical ends $\beta = \frac12dt$ and on the positive cylindrical end $\beta = dt$ for large $|s|$. By \cite[Lemma $16.1$]{Rit2} such a form exists.
\end{itemize}

Let $(H^\pm,J^\pm)$, $H^\pm\in\CL{H}^\pm$
and  $\x^\pm=[x^\pm, \bar{x}^\pm]\in\tPH{\frac12 H^\pm}$, $\y^\pm=[y^\pm, \bar{y}^\pm]\in\tPH{\frac12 H^\pm}$, $\z^\pm=[z^\pm, \bar{z}^\pm]\in\tPH{H^\pm}$. Let $\CL{M}^{PP}\(\x^\pm, \y^\pm, \z^\pm, S, \beta\)$ be the moduli space of smooth maps $u:S\to\hat{M}$, such that

\beq\label{equation:PP moduli space }
\left\{
\begin{array}{l}
du - X_{H^\pm}\otimes\beta\ \text{is}\  (j, J^\pm)\text{-holomorphic}, \\ [0.6em]
u\ \text{converges to $x^\pm$, $y^\pm$  at the negative ends}\\ \text{and converges to $z^\pm$ at the positive end},\\ [0.6em]
(\bar{x}^\pm\cup\bar{y}^\pm)\#u\#(-\bar{z}^\pm)\ \text{represents the zero class in}\ \Gamma.
\end{array}\right.
\eeq

Recall that $du - X_{H^\pm}\otimes\beta$ is $(j, J^\pm)$-holomorphic means that
$$(du - X_{H^\pm}\otimes\beta)^{0,1}:=\frac12 \{(du - X_{H^\pm}\otimes\beta) +  J^\pm\circ(du - X_{H^\pm}\otimes\beta)\circ j\} = 0.$$
On a cylindrical end this becomes Floer's equation $\partial_su + J^\pm(\partial_tu - cX_{H^\pm}) = 0$ for the Hamiltonians $cH^\pm$, where $c\in\{\tfrac12, 1\}$ . For a generic perturbation of $(H^\pm,J^\pm)$ as in \cite[Section $16.5$]{Rit2}, the moduli space $\CL{M}^{PP}\(\x^\pm, \y^\pm, \z^\pm, S, \beta\)$ is a smooth orientable manifold of dimension

$$
\dim\CL{M}^{PP}\(\x^\pm, \y^\pm, \z^\pm, S, \beta\)=\mu(\x^\pm) + \mu(\y^\pm) -\mu(\z^\pm)-2n.
$$

By \cite[Section $16.3$]{Rit2}, we have a sharp energy estimate. Namely, for any $u\in\CL{M}^{PP}\(\x^\pm, \y^\pm, \z^\pm, S, \beta\)$, its energy $E(u)$ is defined by

$$
E(u):=\frac12\int_S\|du - X_{H^\pm}\otimes\beta\|^2\mathrm{vol}_S.
$$
Since $d\beta=0$, we have
\beq\label{equation:sharp energy estimate}
E(u)=\CL{A}_{\frac12 H^\pm}(\x^\pm) + \CL{A}_{\frac12 H^\pm}(\y^\pm) - \CL{A}_{H^\pm}(\z^\pm).
\eeq
By \cite[Section $16.7$]{Rit2} the $\CL{M}^{PP}\(\x^\pm, \y^\pm, \z^\pm, S, \beta\)$ can be compactified by means of broken Floer solutions.

The above setup allows us to repeat the classical arguments from \cite{Sch-Dissertation}, \cite{PSS} and \cite{Seidel} in order to define pair-of-pants products according to the class of Hamiltonians. Namely, for
regular pairs $(H^\pm,J^\pm)$, $H^\pm\in\CL{H}^\pm$,  we have two homomorphisms
\beq\label{equation:pair-of-pants for full Floer}
\ast_{PP}^\pm:HF_k(\tfrac12 H^\pm,J^\pm;\Lambda)\otimes
HF_l(\tfrac12 H^\pm,J^\pm;\Lambda)\to HF_{k+l-2n}(H^\pm,J^\pm;\Lambda).
\eeq
The products $\ast_{PP}^\pm$ are independent of the choices $(\beta, j,J^\pm)$ relative to the ends, see \cite[Theorem $16.10$]{Rit2}.

 \subsection{Morse-theoretical description of quantum homology algebras}\label{subsection: Morse theory}
Let $X$ be  a Liouville  vector field on $M$,  which is defined in some neighborhood of $\partial M$ and which is everywhere transverse to $\partial M$, pointing outwards. Using the flow of  $X$ we can  identify an open neighborhood $U$ of $\partial M$ with $ \partial M \times (-\ve, 0]$ for some $\ve >0$.
\par Let $f\in C^{\infty}(M)$ be a Morse function. Denote by $\Crit_k(f)$ the set of critical points of $f$ of the Morse index $\ind_f(x)=k$, and let $\Crit(f)=\bigcup_{k=0}^{2n}\Crit_k(f)$ be the set of all critical points of $f$. Fix some Riemannian metric $g$ on $M$, and let $x\in\Crit(f)$. Recall that the stable and unstable manifolds of the critical point  $x$ w.r.t. the negative gradient flow $\phi^t:=\phi^t_{-\nabla_gf}$  are the subsets
$$W^s(x;f,g):=\{y\in M|\lim\limits_{t\to+\infty}\phi^t(y)=x\}$$
$$W^u(x;f,g):=\{y\in M|\lim\limits_{t\to-\infty}\phi^t(y)=x\}.$$
We shall consider the special class of Morse functions on $M$ for which the Morse homology algebra is well-defined and isomorphic to the singular homology algebra of $M$. Let $\CL{F}^{+}(g)\subseteq C^{\infty}(M)$  be the set  of all  Morse functions $f$ on $M$, such that
\be
\item[$(U_1)$] $\Crit(f)\cap\cl{U}=\varnothing$ and the gradient vector field $\nabla_{g}f$ of $f$  is everywhere transversal to $\partial \cl{U}$,
\item[$(U_2)$] The negative gradient vector field $-\nabla_{g}f$ points outwards along $\cl{U}$.
\ee
We define also the set $\CL{F}^-(g):=-\CL{F}^+(g):=\{f\in C^{\infty}(M)|-f\in\CL{F}^+(g)\}$. If $f\in\CL{F}^{+}(g)$ and $x\in\Crit(f)$, then the stable manifold $W^s(x;f,g)$ lies in $M\minus\cl{U}$ and it is  diffeomorphic to $\RR^{2n-\ind_f(x)}$. The unstable manifold $W^u(x;f,g)$ is a smooth manifold, possibly with boundary, of dimension $\ind_f(x)$, and  the boundary  $\partial W^u(x;f,g)$ lies in $\partial M$. For a function $f\in\CL{F}^-(g)$ we have a dual picture. A function $f\in\CL{F}^\pm(g)$ will be called an \textbf{\textit{admissible Morse function}}.

Suppose that $f\in\CL{F}^{\pm}(g)$ is Morse-Smale, i.e. the stable and unstable manifolds of $f$ intersect transversally. Then, for any $x,y\in\Crit(f)$ the set of parametrized gradient trajectories $\CL{M}_{x,y}(f,g):=W^u(x;f,g)\pitchfork W^s(y;f,g)$ connecting $x$ and $y$ is a smooth oriented manifold without boundary of dimension $\ind_f(x)-\ind_f(y)$. Note that the intersection $W^u(x;f,g)\pitchfork W^s(y;f,g)$ is indeed well-defined: if $f\in\CL{F}^{+}(g)$ ($f\in\CL{F}^{-}(g)$) then the stable manifold $W^s(y;f,g)$ (the unstable manifold $W^u(x;f,g)$) lies in $M\minus\cl{U}$. The additive group  $\RR$ acts smoothly, freely and properly on $\CL{M}_{x,y}(f,g)$ by reparametrizations. It follows that the space $\hat{\CL{M}}_{x,y}(f,g):=\CL{M}_{x,y}(f,g)/\RR$ of unparametrized gradient trajectories connecting $x$ and $y$
is a smooth oriented manifold  of dimension $\ind_f(x)-\ind_f(y)-1$. If  $\ind_f(x)-\ind_f(y)-1=0$, then $\hat{\CL{M}}_{x,y}(f,g)$ is a finite set of  unparametrized  trajectories, each of which is equipped with an orientation sign $\pm 1$, see  \cite[ Corollary $2.36$]{Sch1}.  Denote by $n_{f,g}(x,y):=\#_{alg}\hat{\CL{M}}_{x,y}(f,g)$ its algebraic number of elements, and define the Morse complex $(CM_*(f),\partial)$ by
$$CM_*(f):=\Crit_*(f)\otimes\FF,\;\;\;\partial(x)=\sum_{\ind_f(x)-\ind_f(y)=1}n_{f,g}(x,y)y.$$

\begin{thm}\label{theorem:Morse homology}(E.g. \cite[ Chapter $4$]{Sch1}; \cite[ Sections $2.7$, $2.8$]{Abbond-Majer}.)\\
Let $f\in\CL{F}^{\pm}(g)$ be a Morse-Smale function.
\item[$(i)$]The Morse complex $(CM_*(f),\partial)$ is a chain complex of $\FF$-vector spaces, i.e. $\partial^2=0$. Its homology  is denoted by $H_*(f;\FF)$.
\item[$(ii)$]For $f\in\CL{F}^{+}(g)$, the Morse homology $H_*(f;\FF)$ is isomorphic to the relative singular homology $H_*(M,\partial M;\FF)$.
\item[$(iii)$]For $f\in\CL{F}^{-}(g)$, the Morse homology $H_*(f;\FF)$ is isomorphic to the  singular homology $H_*(M;\FF)$.
\end{thm}
Next, we describe the intersection products $\bullet_i, i=1,2,3$, in Morse homology, see \cite{Abbond-Schwarz1}, \cite{Abbond-Schwarz2}, \cite{B-C3}. Choose a triple of Morse functions $f_i\in\CL{F}^{-}(g)$, for $i=1,2,3$.  After a generic perturbation of the data $\KL{F}:=\{f_1, f_2, f_3\}$ we may assume that\\

$(\KL{F}_1)$ the functions $f_i$ are Morse-Smale for $i=1,2,3$,\\

$(\KL{F}_2)$ $\Crit(f_1)\cap\Crit(f_2)=\varnothing$,\\

$(\KL{F}_3)$ the triple intersections
\begin{equation}\nonumber
\begin{aligned}
&\CL{M}^1_{x_1,x_2;x_3}\(\KL{F},g \):=W^u(x_1;f_1,g)\cap W^u(x_2;f_2,g)\cap W^s(x_3;f_3,g),\\
&\CL{M}^2_{x_1,x_2;x_3}\(\KL{F},g \):=W^u(x_1;f_1,g)\cap W^u(x_2;-f_2,g)\cap W^s(x_3;f_3,g),\\
&\CL{M}^3_{x_1,x_2;x_3}\(\KL{F},g \):=W^u(x_1;-f_1,g)\cap W^u(x_2;-f_2,g)\cap W^s(x_3;-f_3,g)
\end{aligned}
\end{equation}
are transverse for any $x_i\in\Crit(f_i), i=1,2,3$.

In this case, $\CL{M}^i_{x_1,x_2;x_3}\(\KL{F},g \)$ for $i=1, 2, 3$  are either empty or  smooth oriented manifolds without boundary of dimension
\begin{equation}\nonumber
\begin{aligned}
&\dim\CL{M}^1_{x_1,x_2;x_3}\(\KL{F},g \)=\ind_{f_1}(x_1)+\ind_{f_2}(x_2)-\ind_{f_3}(x_3)-2n,\\
&\dim\CL{M}^2_{x_1,x_2;x_3}\(\KL{F},g \)=\ind_{f_1}(x_1)+\ind_{-f_2}(x_2)-\ind_{f_3}(x_3)-2n,\\
&\dim\CL{M}^3_{x_1,x_2;x_3}\(\KL{F},g \)=\ind_{-f_1}(x_1)+\ind_{-f_2}(x_2)-\ind_{-f_3}(x_3)-2n,
\end{aligned}
\end{equation}
which lie in $M\minus\cl{U}$. The manifolds $\CL{M}^i_{x_1,x_2;x_3}\(\KL{F},g \)$ for $i=1, 2, 3$ are  compact in dimension zero, and thus consist of a finite number of points, each of which is equipped with an orientation sign $\pm 1$. Denote by
$n_{\bullet_i}(x_1, x_2; x_3):=\#_{alg}\CL{M}_{x_1,x_2;x_3}\(\KL{F},g \), i=1, 2, 3$ their algebraic number of elements. Then,\\
$(\bullet_1)$ the $\FF$-bilinear map
\bean
CM_k(f_1)\times CM_l(f_2)&\to CM_{k+l-2n}(f_3),\\ (x_1,x_2)&\mapsto\sum_{\substack{x_3\in\Crit(f_3)\\ \ind_{f_3}(x_3)=k+l-2n}}n_{\bullet_1}(x_1, x_2; x_3)x_3,
\eea
is a chain map that induces the intersection product $\bullet_1$ in  homology,\\
$(\bullet_2)$ the $\FF$-bilinear map
\bean
CM_k(f_1)\times CM_l(-f_2)&\to CM_{k+l-2n}(f_3),\\ (x_1,x_2)&\mapsto\sum_{\substack{x_3\in\Crit(f_3)\\ \ind_{f_3}(x_3)=k+l-2n}}n_{\bullet_2}(x_1, x_2; x_3)x_3,
\eea
is a chain map that induces the intersection product $\bullet_2$ in homology,\\
$(\bullet_3)$ the $\FF$-bilinear map
\bean
CM_k(-f_1)\times CM_l(-f_2)&\to CM_{k+l-2n}(-f_3),\\ (x_1,x_2)&\mapsto\sum_{\substack{x_3\in\Crit(-f_3)\\ \ind_{-f_3}(x_3)=k+l-2n}}n_{\bullet_3}(x_1, x_2; x_3)x_3,
\eea
is a chain map that induces the intersection product $\bullet_3$ in homology.

Suppose that $f\in\CL{F}^{\pm}(g)$ is Morse-Smale and $\Lambda $ is the Novikov ring as in ~\eqref{equation:  Novikov ring}. Define the Morse chain complex $(CM_*(f;\Lambda),\partial_{\Lambda})$ with coefficients in  the Novikov ring $\Lambda $ by
$$
CM_* (f;\Lambda):=\Crit_*(f)\otimes\Lambda,\;\;\partial_{\Lambda}(x\otimes\lambda)=
\sum_{\ind_f(x)-\ind_f(y)=1}n_{f,g}(x,y)y\otimes\lambda.
$$
The grading is given by $\deg(x\otimes zs^{\alpha}q^m)=\ind_f(x)+2m.$  The Morse homology $H_*(f;\Lambda)$  is isomorphic as a $\Lambda$-module either to the relative quantum homology  $QH_*(M,\partial M)$ for $f\in\CL{F}^{+}(g)$ or to the absolute quantum homology $QH_*(M)$ for $f\in\CL{F}^{-}(g)$. Let us  describe the quantum products $\ast_i$ for $i=1,2,3$ via the Morse homology. For a class $A\in\Gamma$ -- see \eqref{equation: Gamma group of spherical classes} --  and  pairwise distinct marked points $\mathbf{z}:=(z_1, z_2, z_3)\in\(\mathbb{S}^2\)^3,$ choose a generic pair $(J,\KL{H})\in\CL{JH}_{reg}(M,\partial M,\omega,A,\mathbf{z})$  of an almost complex structure and a Hamiltonian perturbation of the Cauchy-Riemann section -- see \cite[Sections 2.1.8-2.1.9]{Lanzat-Thesis}. Let $g_J$ be the Riemannian metric induced by $J$. Take a  generic Morse data
$\KL{F}:=\{f_1, f_2, f_3\}\subseteq\CL{F}^{-}(g_J)$ in the sense of  $(\KL{F}_i), i=1,2,3$, and take their critical points $x_i\in\Crit(f_i), i=1,2,3$. Define the following subspaces of the space $\CL{M}(A,J,\KL{H})$ of  $(J,\KL{H})$-holomorphic $A$-spheres:\\
$(\CL{M}^1)$ the space
\bean
\CL{M}^1_{x_1,x_2;x_3}&\(A,\KL{F},g_J \):=\\
&\(\mathrm{ev}_{\mathbf{z},J,\Kl{H}}\)^{-1}\(W^u(x_1;f_1,g_J)\times W^u(x_2;f_2,g_J)\times W^s(x_3;f_3,g_J) \)
\eea
of all  $(J,\KL{H})$-holomorphic $A$-spheres $u:\CC\PP^1\to M$ such that $$u(z_1)\in W^u(x_1;f_1,g_J),\;\;\;u(z_2)\in W^u(x_2;f_2,g_J),\;\;\;u(z_3)\in W^s(x_3;f_3,g_J),$$
$(\CL{M}^2)$ the space
\bean
\CL{M}^2_{x_1,x_2;x_3}&\(A,\KL{F},g_J \):=\\
&\(\mathrm{ev}_{\mathbf{z},J,\Kl{H}}\)^{-1}\(W^u(x_1;f_1,g_J)\times W^u(x_2;-f_2,g_J)\times W^s(x_3;f_3,g_J) \)
\eea
of all  $(J,\KL{H})$-holomorphic $A$-spheres $u:\CC\PP^1\to M$ such that $$u(z_1)\in W^u(x_1;f_1,g_J),\;\;\;u(z_2)\in W^u(x_2;-f_2,g_J),\;\;\;u(z_3)\in W^s(x_3;f_3,g_J),$$
$(\CL{M}^3)$ the space
\bean
\CL{M}^3_{x_1,x_2;x_3}&\(A,\KL{F},g_J \):=\\ &\(\mathrm{ev}_{\mathbf{z},J,\Kl{H}}\)^{-1}\(W^u(x_1;-f_1,g_J)\times W^u(x_2;-f_2,g_J)\times W^s(x_3;-f_3,g_J) \)
\eea
of all  $(J,\KL{H})$-holomorphic $A$-spheres $u:\CC\PP^1\to M$ such that $$u(z_1)\in W^u(x_1;-f_1,g_J),\;\;\;u(z_2)\in W^u(x_2;-f_2,g_J),\;\;\;u(z_3)\in W^s(x_3;-f_3,g_J).$$
Here, $\mathrm{ev}_{\mathbf{z},J,\Kl{H}}:\CL{M}(A,J,\KL{H}) \to M^3$ is the evaluation map given by $\mathrm{ev}_{\mathbf{z},J,\Kl{H}}(u)=(u(z_1),u(z_2),u(z_3))$. Since the chosen data is generic, the above spaces are smooth oriented manifolds of dimension
\begin{equation}\nonumber
\begin{aligned}
&\dim\CL{M}^1_{x_1,x_2;x_3}\(A,\KL{F},g_J \)=2c_1(A)-2n+\ind_{f_1}(x_1)+\ind_{f_2}(x_2)-\ind_{f_3}(x_3),\\
&\dim\CL{M}^2_{x_1,x_2;x_3}\(A,\KL{F},g_J \)=2c_1(A)-2n+\ind_{f_1}(x_1)+\ind_{-f_2}(x_2)-\ind_{f_3}(x_3),\\
&\dim\CL{M}^3_{x_1,x_2;x_3}\(A,\KL{F},g_J \)=2c_1(A)-2n+\ind_{-f_1}(x_1)+\ind_{-f_2}(x_2)-\ind_{-f_3}(x_3).
\end{aligned}
\end{equation}
They  are  compact in dimension zero, and thus consist of a finite number of points, each of which is equipped with an orientation sign $\pm 1$. Denote by
$n^A_{\ast_i}(x_1, x_2; x_3):=\#_{\ff}\CL{M}^i_{x_1,x_2;x_3}\(A,\KL{F},g_J \), i=1, 2, 3$, their algebraic (over $\FF$) number of elements. We then define\\
$(\ast_1)$ the $\Lambda$-bilinear map $CM_k(f_1;\Lambda)\times CM_l(f_2;\Lambda)\to CM_{k+l-2n}(f_3;\Lambda)$
$$(x_1\otimes\lambda_1,x_2\otimes\lambda_2)
\mapsto\sum_{A\in\Gamma}(x_1\ast_1 x_2)_A\otimes s^{-\omega(A)}q^{-2c_1(A)}\lambda_1\lambda_2,$$
where
$$(x_1\ast_1 x_2)_A:= \sum_{\substack{x_3\in\Crit(f_3)\\ \dim\Cl{M}^1_{x_1,x_2;x_3}\(A,\Kl{F},g_J \)=0}}n^A_{\ast_1}(x_1, x_2; x_3)x_3,$$
is a chain map that induces the intersection product $\ast_1$ in quantum  homology,\\
$(\ast_2)$ the $\Lambda$-bilinear map $CM_k(f_1;\Lambda)\times CM_l(-f_2;\Lambda)\to CM_{k+l-2n}(f_3;\Lambda)$
$$
(x_1\otimes\lambda_1,x_2\otimes\lambda_2)
\mapsto\sum_{A\in\Gamma}(x_1\ast_2 x_2)_A\otimes s^{-\omega(A)}q^{-2c_1(A)}\lambda_1\lambda_2,$$
where
$$(x_1\ast_2 x_2)_A:= \sum_{\substack{x_3\in\Crit(f_3)\\ \dim\Cl{M}^2_{x_1,x_2;x_3}\(A,\Kl{F},g_J \)=0}}n^A_{\ast_2}(x_1, x_2; x_3)x_3,$$
is a chain map that induces the intersection product $\ast_2$ in quantum  homology,\\
$(\ast_3)$ the $\Lambda$-bilinear map $CM_k(-f_1;\Lambda)\times CM_l(-f_2;\Lambda)\to CM_{k+l-2n}(-f_3;\Lambda)$
$$
(x_1\otimes\lambda_1,x_2\otimes\lambda_2)
\mapsto\sum_{A\in\Gamma}(x_1\ast_3 x_2)_A\otimes s^{-\omega(A)}q^{-2c_1(A)}\lambda_1\lambda_2,$$
where
$$(x_1\ast_3 x_2)_A:= \sum_{\substack{x_3\in\Crit(-f_3)\\ \dim\Cl{M}^3_{x_1,x_2;x_3}\(A,\Kl{F},g_J \)=0}}n^A_{\ast_3}(x_1, x_2; x_3)x_3,$$
is a chain map that induces the intersection product $\ast_3$ in quantum  homology.

\subsection{The PSS isomorphisms}
Recall that if  $f^\pm\in\CL{F}^\pm(g)$ are admissible Morse-Smale functions, then the Morse homology $H_*(f^+;\Lambda)$  is isomorphic to the relative quantum homology $QH_*(M,\partial M)$ as an $\Lambda$-algebra, and the Morse homology $H_*(f^-;\Lambda)$  is isomorphic to the absolute quantum homology $QH_*(M)$ as an $\Lambda$-algebra. Using these isomorphisms let $\ast^+:H_*(f^+;\Lambda)\otimes H_*(f^+;\Lambda)\to H_*(f^+;\Lambda)$
be the corresponding quantum product  $\ast_3$ on $QH_*(M,\partial M)$, and let
$\ast^-:H_*(f^-;\Lambda)\otimes H_*(f^-;\Lambda)\to H_*(f^-;\Lambda)$
be the corresponding quantum product  $\ast_1$ on $QH_*(M)$, see \eqref{equation: quantum intersection products}.
Following the arguments of U. Frauenfelder and F. Schlenk in \cite[ Section $4$]{F-S} we shall construct the Piunikhin-Salamon-Schwarz-type isomorphisms of $\Lambda$-algebras
$$
\Phi_{PSS}^\pm : \(H_*(f^\pm;\Lambda),\ast^\pm\)\to\(HF_*\(H^\pm,J^\pm;\Lambda\), \ast_{PP}^\pm\),
$$
see also \cite{PSS}; \cite[Section $12.1$]{MS3}.

Given two admissible almost complex structures $J^{-\infty}, J^{+\infty}\in\CL{J}_{\scriptstyle\mathbb{S}^1, P_{-2\ve}}$, consider  the space $\CL{J}_{\scriptstyle\mathbb{S}^1, P_{-2\ve}}(J^-,J^+)$ of smooth families of admissible almost complex structures, such that  $(J_s)_{s\in\rr} \in \CL{J}_{\scriptstyle\mathbb{S}^1, P_{-2\ve}}(J^{-\infty},J^{+\infty})$ if and only if $J_s\in\CL{J}_{\scriptstyle\mathbb{S}^1, P_{-2\ve}}$ for all $s\in\RR$ and there exists $s_0=s_0(J_s)>0$ such that $J_s=J^{-\infty}$ for $s \leq -s_0$ and $J_s=J^{+\infty}$ for $s \geq s_0$.

Consider regular admissible Hamiltonian $H:=H^+\in\CL{H}^+_{reg}$  and admissible Morse-Smale functions $f^+\in\CL{F}^+(g)$. There exists $\hat{H}\in\hCL{H}^+$,
such that $H = \hat{H}|_{\scriptstyle\mathbb{S}^1 \times M}$ and  $\hat{H}|_{\scriptstyle\mathbb{S}^1 \times\cl{P_0} } = h\circ e$ for some $h \in C^\infty(\RR)$ satisfying $0 \leq -h'(e^r) < \kappa$ for all $r\ge 0$. Take a smooth homotopy $\(h_s\)_{s\in\rr}\subseteq C^\infty( \RR )$, such that
\begin{itemize}
 \item[(h1)] $h_s=0, \quad s \le 0$,
 \item[(h2)] $\partial_sh_s'(e^r)\le 0$  for all $r\ge 0$ and all $s\in\mathbb{R}$,
 \item[(h3)] $h_s = h, \quad s \ge 1$,
\end{itemize}
and then choose a smooth homotopy
$(\hat{H}_s)_{s\in\rr}\subseteq C^\infty\(\mathbb{S}^1 \times \hat{M}\),$ such that
\begin{itemize}
 \item[(H1)] $\hat{H}_s = 0, \quad s \le 0$,
 \item[(H2)] $\hat{H}_s|_{\scriptstyle\mathbb{S}^1 \times \cl{P_0}} =  h_s \circ e, \quad 0\leq s\leq 1$,
 \item[(H3)] $\hat{H}_s=\hat{H}, \quad s \ge 1$.
\end{itemize}
We define  $H_s \in C^\infty(\mathbb{S}^1 \times M)$ to be $H_s:=\hat{H}_s|_{\scriptstyle\mathbb{S}^1 \times M}.$ The following important theorem was proved in \cite[Theorem $4.1$]{F-S}, see also \cite[Example $3.3$]{PSS}.
\begin{thm}  \label{thm: PSS half cyclinders}
Let $J^{+\infty}\in\CL{J}_{\scriptstyle\mathbb{S}^1, P_{-2\ve}}$ be an $H$-regular admissible almost complex structure and $J^{-\infty}\in\CL{J}_{\scriptstyle\mathbb{S}^1, P_{-2\ve}}$ be an arbitrary admissible almost complex structure.  For any $\x=[x,\bar{x}]\in\tPH{H}$ and a generic element $(J_s)_{s\in\rr} \in \CL{J}_{\scriptstyle\mathbb{S}^1, P_{-2\ve}} (J^{-\infty},J^{+\infty})$ the moduli space $\CL{M}\(\x,(H_s)_{s\in\rr}, (J_s)_{s\in\rr}\)$ of the problem
\beq\label{equation:PSS half cyclinders}
\left\{
\begin{array}{l}
u \in  C^\infty\(\RR\times\mathbb{S}^1,\hat{M}\), \\ [0.4em]
\partial_s u+\hat{J}_{s,t}(u) \left( \partial_t u-X_{\hat{H}_{s,t}}(u) \right) =0, \\ [0.4em]
\int_{\rr \times\scriptstyle\mathbb{S}^1} \left| \partial_s u \right|^2 < \infty,\\ [0.4em]
\lim\limits_{s\to+\infty}u(s,t)=x(t),\\ [0.4em]
u\#(-\bar{x})\ \text{represents the zero class in}\ \Gamma.
\end{array}\right.
\eeq
is a smooth manifold of dimension $2n-\mu(\x)$.
Here, $\hat{J}_{s,t}=\hat{J}_s(\cdot,t)$ is the extension of $J_s(\cdot,t)$ and $\hat{H}_{s,t}=\hat{H}_s(\cdot,t)$.
\end{thm}
\begin{rem}\label{remark:  floer plains}
The finite energy condition, the condition $(H1)$ and the removable of singularities theorem imply that the limit $\lim_{s\to-\infty}u(s,t)$ exists, which is a point in $\hat{M}$. By the maximum principle, see Corollary~\ref{maximum}, it actually lies in $M\minus\partial M$. Moreover, since the homotopy $h_s$ satisfies  the condition $(h2)$, the maximum principle, see Corollary~\ref{maximum} - $(ii)$, implies that the map $u\#(-\bar{x})$ is a sphere in $M\minus\partial M$ for any solution $u\in\CL{M}\(\x,
(H_s)_{s\in\rr}, (J_s)_{s\in\rr}\)$.
\end{rem}

\noindent We have a well-defined smooth evaluation map $\ev:\CL{M}\(\x,(H_s)_{s\in\rr}, (J_s)_{s\in\rr}\)\to M$, $\ev(u)=\lim_{s\to-\infty}u(s,t)$. For  generic family $(J_s)_{s\in\rr}$ the map $\ev$ is transversal to every unstable manifold $W^u(p;f^+,g)$, $p\in\Crit(f^+)$. Hence for every $p\in\Crit(f^+)$ and every $\x=[x,\bar{x}]\in\tPH{H}$ the moduli space of mixed trajectories
\beq\label{equation: Phi PSS }
\CL{M}^{p\rightsquigarrow \x}_+:=\left\{(\gamma,u)\;\,\vline\;
\begin{aligned}
&\gamma\in C^{\infty}((-\infty,0],M),\\
&u\in\CL{M}(\x,(H_s)_{s\in\rr}, (J_s)_{s\in\rr}),\\
&\gamma^\prime(s)=-\nabla_gf^+(\gamma(s)), \lim_{s\to-\infty}\gamma(s)=p,\\
&\gamma(0)=\ev(u), \lim_{s\to+\infty}u(s,t)=x(t),\\
&u\#(-\bar{x})\ \text{represents the zero class in}\ \Gamma.
\end{aligned}
\right\}
\eeq
is a smooth manifold of dimension $\ind_{f^+}(p)-\mu(\x)$. It carries an orientation, see \cite{FH}. When $\ind_{f^+}(p)=\mu(\x)$ the manifold $\CL{M}^{p\rightsquigarrow \x}_+$ is zero dimensional, and by strong semi-positivity its compact, i.e. a finite set of oriented points. In this case denote by $n^{PSS}_+(p, \x):=\#_{\ff}\CL{M}^{p\rightsquigarrow \x}_+$ the algebraic (over $\FF$) number of its elements. Following \cite{PSS} and \cite{F-S} we define a $\Lambda$-module homomorphism
$
\phi_{PSS}^+:CM_*(f^+;\Lambda)\to CF_*(H^+, J^{+\infty};\Lambda)
$
by the $\Lambda$-linear extension of
\beq\label{equation: Phi PSS+ chain map}
\phi_{PSS}^+(p\otimes 1):=\sum_{\substack{\x\in\tph{H}\\\ind_{f^+}(p)=\mu(\x)}}n^{PSS}_+(p, \x)\x\otimes 1\;(\mathrm{mod}\  R).
\eeq
By  \cite[Theorem $3.3$]{H-S} and \cite[Corollary $3.4$]{Seidel},  the map is indeed a well-defined $\Lambda$-module homomorphism, which respects the grading. Recall that for the Morse chain complex $(CM_*(f^+;\Lambda),\partial_{\Lambda})$ with coefficients in  the Novikov ring $\Lambda $ we have $$\partial_{\Lambda}(p\otimes\lambda)=\sum_{\substack{q\in\Crit(f^+)\\
\ind_{f^+}(p)-\ind_{f^+}(q)=1}}n_{f^+,g}(p,q)q\otimes\lambda.$$
Since $(M,\omega)$ is strongly semi-positive, we may apply the standard gluing and compactness arguments to get that for every $p\in\Crit(f^+),\ \x=[x,\bar{x}]\in\tPH{H}$ such that $\ind_{f^+}(p)-\mu(\x)=1$
$$
\sum_{\substack{q\in\Crit(f^+)\\\ind_{f^+}(p)-\ind_{f^+}(q)=1\\\ind_{f^+}(q)=\mu(\x)}} n_{f^+,g}(p,q)n^{PSS}_+(q, \x)=
\sum_{\substack{\y\in\tph{H}\\\ind_{f^+}(p)=\mu(\y)\\ \mu(\y)-\mu(\x)=1}} n^{PSS}_+(p, \y)n_{H^+,J^{+\infty}}(\y,\x).
$$
It follows that $\phi_{PSS}^+$ intertwines the Morse and the Floer boundary operators and hence it induces a homomorphism of $\Lambda$-modules
$
\Phi_{PSS}^+ :H_*(f^+;\Lambda)\to HF_*\(H^+,J^{+\infty};\Lambda\).
$
In order to show that $\Phi_{PSS}^+$ is a $\Lambda$-module isomorphism we construct its opposite homomorphism of $\Lambda$-modules
$
\Psi^{PSS}_+ :HF_*\(H^+,J^{+\infty};\Lambda\)\to H_*(f^+;\Lambda).
$
For that matter we define the chain-homotopic inverse
$
\psi^{PSS}_+: CF_*(H^+,J^{+\infty};\Lambda)\to CM_*(f^+;\Lambda)
$
of $\phi_{PSS}^+$ as explained in \cite[Section $4$]{F-S}. Note that the closure
$$
\overline{\bigcup_{x\in\Crit(f^+)}W^s(x;f^+,g)}
$$
of the union of stable submanifolds of $f^+$ lies in $M\minus\partial M$. Choose an open neighborhood $V\subset M\minus\partial M$ of this closure.  Then choose a smooth homotopy
$(\hat{H}_s)_{s\in\rr}\subseteq \hCL{H}^+,$ such that for the restriction  $H_s:=\hat{H}_s|_{\scriptstyle\mathbb{S}^1 \times M}$  there exits $s_0<0$ for which we have
$$
H_s\big|_V = 0,\quad \text{if}\ s \le s_0\ \text{and}\ H_s=H,\quad \text{if}\ s \ge -s_0.
$$
As before, $\hat{H_s}|_{\scriptstyle\mathbb{S}^1 \times\cl{P_0} } = h_s\circ e$ for some $h_s \in C^\infty(\RR)$ satisfying $0 \leq -h_s'(e^r) < \kappa$ for all $s\in\mathbb{R}$ and all $r\ge 0$. Assume, in addition, that the homotopy $\(h_s\)_{s\in\rr}\subseteq C^\infty( \RR )$ is constant $h\in C^\infty( \RR )$, i.e. independent of the $s$-parameter and that $h'(1)<0$.

Now, for any $\x=[x,\bar{x}]\in\tPH{H}$ and a generic element $(J_{-s})_{s\in\rr} \in \CL{J}_{\scriptstyle\mathbb{S}^1, P_{-2\ve}} (J^{-\infty},J^{+\infty})$ the moduli space $\CL{M}\(\x,(H_{-s})_{s\in\rr}, (J_{-s})_{s\in\rr}\)$ of the problem
\beq\label{equation:PSS oposite half cyclinders}
\left\{
\begin{array}{l}
u \in  C^\infty\(\RR\times\mathbb{S}^1,\hat{M}\), \\ [0.4em]
\partial_s u+\hat{J}_{-s,t}(u) \left( \partial_t u-X_{\hat{H}_{-s,t}}(u) \right) =0, \\ [0.4em]
\int_{\rr \times\scriptstyle\mathbb{S}^1} \left| \partial_s u \right|^2 < \infty,\\ [0.4em]
\lim\limits_{s\to-\infty}u(s,t)=x(t),\\ [0.4em]
u\#(-\bar{x})\ \text{represents the zero class in}\ \Gamma.
\end{array}\right.
\eeq
is a smooth manifold of dimension $\mu(\x)$. As above, we have a well-defined smooth evaluation map $$\ev:\CL{M}\(\x,(H_{-s})_{s\in\rr}, (J_{-s})_{s\in\rr}\)\to M\minus\partial M,\;
\ev(u)=\lim_{s\to+\infty}u(s,t).$$  For  generic family $(J_{-s})_{s\in\rr}$ the map $\ev$ is transversal to every stable manifold $W^s(p;f^+,g)$, $p\in\Crit(f^+)$. Hence for every $p\in\Crit(f^+)$ and every $\x=[x,\bar{x}]\in\tPH{H}$ the moduli space of mixed trajectories
\beq\label{equation: Psi PSS }
\CL{M}^{\x \rightsquigarrow p}_+:=\left\{(u,\gamma)\;\vline\;
\begin{aligned}
&u\in\CL{M}\(\x,(H_{-s})_{s\in\rr}, (J_{-s})_{s\in\rr}\),\\
&\gamma\in C^{\infty}([0,+\infty),M),\\
&\gamma^\prime(s)=-\nabla_gf^+(\gamma(s)), \lim_{s\to+\infty}\gamma(s)=p,\\
&\gamma(0)=\ev(u), \lim_{s\to-\infty}u(s,t)=x(t),\\
&\bar{x}\#u\ \text{represents the zero class in}\ \Gamma.
\end{aligned}
\right\}
\eeq
is a smooth manifold of dimension $\mu(\x)-\ind_{f^+}(p)$. It carries an orientation. When $\ind_{f^+}(p)=\mu(\x)$ the manifold $\CL{M}^{\x \rightsquigarrow p}_+$ is zero dimensional, and by strong semi-positivity its compact, i.e. a finite set of oriented points. In this case denote by $n^{PSS}_+(\x, p):=\#_{\ff}\CL{M}^{\x \rightsquigarrow p}_+$ the algebraic (over $\FF$) number of its elements. Following \cite{PSS} and \cite{F-S} we define a $\Lambda$-module homomorphism
$
\psi^{PSS}_+:CF_*(H^+,J^{+\infty};\Lambda)\to CM_*(f^+;\Lambda)
$
by the $\Lambda$-linear extension of
\beq\label{equation: Psi PSS+ chain map}
\psi^{PSS}_+(\x\otimes 1):=\sum_{\substack{p\in\Crit(f^+)\\\ind_{f^+}(p)=\mu(\x)}}n^{PSS}_+(\x, p)p\otimes 1.
\eeq

By  \cite[Theorem $3.3$]{H-S} and \cite[Corollary $3.4$]{Seidel}, the map is indeed a well-defined $\Lambda$-module homomorphism, which respects the grading.  Applying the standard gluing and compactness arguments for strongly semi-positive symplectic manifolds we conclude for every $p\in\Crit(f^+),\ \x=[x,\bar{x}]\in\tPH{H}$ such that $\mu(\x)-\ind_{f^+}(p)=1$
$$
\sum_{\substack{q\in\Crit(f^+)\\\ind_{f^+}(q)-\ind_{f^+}(p)=1\\\ind_{f^+}(q)=
\mu(\x)}}n^{PSS}_+(\x,q) n_{f^+,g}(q,p)=
\sum_{\substack{\y\in\tph{H}\\\ind_{f^+}(p)=\mu(\y)\\ \mu(\x)-\mu(\y)=1}} n_{H^+,J^{+\infty}}(\x,\y)n^{PSS}_+(\y,p).
$$
It follows that $\psi^{PSS}_+$ intertwines the Morse and Floer boundary operators and hence it induces a homomorphism of $\Lambda$-modules
$\Psi^{PSS}_+ :H_*(f^+;\Lambda)\to HF_*\(H^+,J^{+\infty};\Lambda\)$. Following the arguments from \cite[Theorems $4.1$ and $5.1$]{PSS}; \cite[Sections $12.1-12.2$]{MS3}; \cite[pages $2350-2351$ ]{Alb},  we conclude that
\bean
&\Phi_{PSS}^+ : \(H_*(f^+;\Lambda),\ast^+\)\to\(HF_*\(H^+,J^{+\infty};\Lambda\), \ast_{PP}^+\)\\
&\Psi_{PSS}^+ :\(HF_*\(H^+,J^{+\infty};\Lambda\), \ast_{PP}^+\) \to\(H_*(f^+;\Lambda),\ast^+\)
\eea
are mutually opposite  $\Lambda$-algebras isomorphisms.

Repeating the above constructions for the opposite data:\\
$\bullet$ the Morse-Smale function $f^-:=-f^+\in\CL{F}^-(g)$,\\
$\bullet$ the regular Hamiltonian $H^-:=H^{(-1)}$, where $H^{(-1)}(t,x):=-H(-t,x)\in\CL{H}^-_{reg}$,\\
$\bullet$ the generic family $(J^{(-1)}_s)_{s\in\rr} \in \CL{J}_{\scriptstyle\mathbb{S}^1, P_{-2\ve}} \(\(J^{-\infty}\)^{(-1)},\(J^{+\infty}\)^{(-1)}\)$, such that $\(J^{+\infty}\)^{(-1)}$ is $H^-$-regular, where $J_s^{(-1)}(t,p):=J_s(-t,p)$,\\
we get  mutually opposite  $\Lambda$-algebras isomorphisms
\bean
&\Phi_{PSS}^- : \(H_*(f^-;\Lambda),\ast^-\)\to\(HF_*\(H^-,\(J^{+\infty}\)^{(-1)};\Lambda\), \ast_{PP}^-\)\\
&\Psi_{PSS}^- :\(HF_*\(H^-,\(J^{+\infty}\)^{(-1)};\Lambda\), \ast_{PP}^-\) \to\(H_*(f^-;\Lambda),\ast^-\).
\eea

\section{Spectral invariants}\label{subsection: spectral invariants}
Recall that $\Ham_c(M,\omega)$ is the group of  smooth compactly supported Hamiltonian diffeomorphisms of $(M,\omega)$, i.e.  the group of time-1-maps $\phi^1_H$ $$\Ham_c(M,\omega):=\{\phi=\phi^1_H|\; H\in \CL{H}_c(M)\},$$ where
$\phi^t_{H}$ is the flow generated by the time-dependent Hamiltonian vector field $X_{H_t}$ of $H$ and $\CL{H}_c(M)$ is the space of $C^\infty$-smooth compactly supported functions on $\mathbb{S}^1\times M$. Denote by $\tHam_c(M,\omega)$ the universal cover of $\Ham_c(M,\omega)$, i.e. the set of homotopy classes
$\til{\phi}:=\til{\phi^1_H}:=\left[\phi(s)\right]_{\text{rel}\ \{\id,\phi(1)=\phi^1_H\}}$ relatively to fixed ends of smooth identity-based paths $$s\mapsto\phi(s)\in C^{\infty}\(([0,1],\{0\}),(\Ham_c(M,\omega),\id)\).$$ We shall write $H\sim K$ if $\til{\phi^1_H}=\til{\phi^1_K}$. Denote also $\(QH^+_*, \ast^+\):=
\(QH_*(M,\partial M;\Lambda), \ast_3\)$ and $\(QH^-_*, \ast^-\):= \(QH_*(M;\Lambda), \ast_1\)$.

Following the papers of Viterbo \cite{V1}, Schwarz \cite{Sch} and Oh \cite{Oh1} we define two kinds of spectral invariants
\beq
c^\pm: QH^\pm_*\times\CL{H}_c(M)\to\RR,
\eeq
which descend to
\beq
c^\pm: QH^\pm_*\times\tHam_c(M,\omega)\to\RR
\eeq
as follows. First, for a regular admissible pair $(H^\pm, J^\pm)$, with $H^\pm\in\CL{H}_{reg}^\pm$,  and for $0\neq a^\pm\in QH^\pm_*$ we define
\beq\label{equation:selectors in H^pm_reg}
c^\pm(a^\pm, H^\pm):=\inf\left\{\alpha\in\RR\ | \ \Phi_{PSS}^\pm(a^\pm)\in HF_*^{(-\infty,\alpha)}(H^\pm, J^\pm;\Lambda)\right\}.
\eeq
We note that the equivalent definition of the spectral numbers is
\beq\label{equation:selectors in H^pm_reg via ell}
c^\pm(a^\pm, H^\pm):=\sup\left\{\ell_{H^\pm}\(c^\pm\)|\ c^\pm\in CF_*(H^\pm;\Lambda), [c^\pm]=\Phi_{PSS}^\pm(a^\pm)\right\}.
\eeq
In particular,
\beq\label{equation: selectors in H^pm_reg via a^[m]}
c^\pm(a^\pm, H^\pm):=\sup_m c^\pm\(\(a^\pm\)^{[m]}, H^\pm\),
\eeq
where $\(a^\pm\)^{[m]}$ is the degree-$m$-component of $a^\pm$.

For the next proposition we need the following refinement. Denote by $\CL{H}_{lin}^\pm$ the subspace of $\CL{H}^\pm$ consisting of Hamiltonians with linear completions. Namely, $H^\pm\in\CL{H}_{lin}^\pm$ if and only if the corresponding completion $\hat{H}^\pm$ restricted to the tube $P_{-2\ve}$ has a form
$$
\hat{H}^\pm\big|_{P_{-2\ve}}(x,r) = A^\pm e^r + B^\pm,\; A^\pm, B^\pm\in\RR.
$$
Note that the slope $A^\pm$ must satisfy
\beq\label{equation:small slopes}
0\leq\mp A^\pm<\kappa.
\eeq
Denote also $\CL{H}_{lin, reg}^\pm:=\CL{H}_{lin}^\pm\cap\CL{H}_{reg}^\pm$.

\begin{prop}\label{prop: well definition of c^pm}
\mbox{}\\
$(i)$ Spectral numbers $c^\pm(a^\pm, H^\pm)$ are finite.\\
$(ii)$ Spectral numbers $c^\pm(a^\pm, H^\pm)$ do not depend on the choice of $H^\pm$-regular almost complex structures $J^\pm$ .\\
$(iii)$ For $H^\pm,K^\pm\in\CL{H}_{lin, reg}^\pm$, we have
$$
c^\pm(a^\pm, H^\pm)-c^\pm(a^\pm, K^\pm)\leq\int_0^1\max_{p\in M}(H^\pm(t,p)-K^\pm(t,p))dt
$$
and
$$c^\pm(a^\pm, H^\pm)-c^\pm(a^\pm, K^\pm)\geq\int_0^1\min_{p\in M}(H^\pm(t,p)-K^\pm(t,p))dt.
$$
It follows that
$$|c^\pm(a^\pm, H^\pm)-c^\pm(a^\pm, K^\pm)|\leq\hn{H^\pm-K^\pm},
$$
where the \emph{\textbf{$L^{(1,\infty)}$-norm}} $\hn{\cdot}$ on $ C^{\infty}([0,1]\times M)$ is defined as
 $$\hn{H}:=\int\limits^1_0\left(\sup\limits_{x\in M}H(t,x)-\inf\limits_{x\in M}H(t,x)\right)dt.$$
In particular, the functions $H^\pm\mapsto c^\pm(a^\pm, H^\pm)$ are $C^0$-continuous on $\CL{H}_{lin, reg}^\pm$.
\end{prop}
\begin{proof}
The proof of items $(i)$ and $(ii)$ follows verbatim the proof of \cite[Theorem $5.3$]{Oh Spec for general1}. The proof of $(iii)$ follows verbatim the proof of \cite[Proposition $5.8$]{Oh Spec for general1} once the corresponding continuation homomorphisms
$$
HF_*(H^\pm,J^\pm;\Lambda)\to HF_*(K^\pm,J^\pm;\Lambda)
$$
are isomorphisms. By the choice of slopes of the completions, see \eqref{equation:small slopes}, this indeed holds, see \cite[Lemma $11$]{Rit3}.
\end{proof}

\begin{cor}\label{cor:extension of selectors to compact support}
The functions $c^\pm:H^\pm\mapsto c^\pm(a^\pm, H^\pm)$ can be $C^\infty$-continuously extended to  functions $c^\pm:\CL{H}_{lin}^\pm\to\RR$, which are 1-Lipschitz w.r.t. the $L^{(1,\infty)}$-norm. In particular, $c^\pm(a^\pm, H)$ are well defined for $C^\infty$-smooth compactly supported $H\in\CL{H}_c(M)\subset\CL{H}_{lin}^+\cap\CL{H}_{lin}^-$.
\end{cor}
\begin{proof}
Since $\CL{H}_{lin, reg}^\pm$ are $C^\infty$-dense in $\CL{H}_{lin}^\pm$, define
$$
c^\pm(a^\pm, H^\pm):=\lim_{n\to+\infty}c^\pm(a^\pm, H_n^\pm),
$$
for any $H^\pm\in\CL{H}_{lin}^\pm$, where
$\CL{H}_{lin, reg}^\pm\ni H_n^\pm\overset{C^\infty}{\longrightarrow}H^\pm$.
\end{proof}

\begin{prop}\label{prop: properties of c_pm on H_reg}
For any $H^\pm, K^\pm\in\CL{H}_{lin}^\pm$ and any $0\neq a^\pm, b^\pm\in QH^\pm_*$ we have the following properties of spectral numbers.

{\bf (Spectrality)}\ If $(M,\omega)$ is rational, i.e. the group $\displaystyle \omega\(H_2^S(M)\)\leq\RR$ is a discrete subgroup of $\RR$, or if $(M,\omega)$ is irrational, but the Hamiltonians $H^\pm$ are non-degenerate, then $c^\pm(a^\pm, H^\pm)\in \Spec (H^\pm)$.

{\bf (Quantum homology shift property)}\ $c^\pm (\lambda a^\pm, H^\pm)
= c^\pm (a^\pm, H^\pm) + \nu (\lambda)$ for all $\lambda \in \Lambda$, where
$\nu$ is the valuation from Definition~\ref{definition: Novikov ring Lambda}.

{\bf (Monotonicity)}\ If $H^\pm\leq K^\pm$, then $c^\pm
(a^\pm, H^\pm)\leq c^\pm (a^\pm, K^\pm)$.

{\bf ($ C^0$-continuity)}\ $|c^\pm(a^\pm, H^\pm)-c^\pm(a^\pm, K^\pm)|\leq\hn{H^\pm-K^\pm}.$

{\bf (Symplectic invariance)}\ $c^\pm(a^\pm,\psi^*H^\pm) = c^\pm (a^\pm,H^\pm)$
for every $\psi \in \Symp^0_c (M,\omega)$.

{\bf (Normalization)}\ $c^\pm (a^\pm,0) = \nu (a^\pm)$ for every $a^\pm\in QH^\pm_*$, see Definition~\ref{definition: Novikov ring Lambda}.

{\bf (Homotopy invariance)}\ If $H, K\in\CL{H}_c(M)$ and $H\sim K$ then  $c^\pm (a^\pm, H) = c^\pm (a^\pm, K)$.  Thus one can define
$c^\pm (a^\pm,\til{\phi})$ for any $\til{\phi}\in\tHam_c (M, \omega)$ as $c^\pm (a^\pm,H)$ for any  $H\in\CL{H}_c(M)$
generating $\til{\phi}$, i.e. $\til{\phi}=\til{\phi^1_H}$.

{\bf (Triangle inequality)}\ For any  $H, K\in\CL{H}_c(M)$ then
$$c^\pm (a^\pm\ast^\pm b^\pm,H\#K)\leq c^\pm (a^\pm, H) + c^\pm (b^\pm, K),$$ and thus, for any $\til{\phi}, \til{\psi}\in\tHam_c (M, \omega)$ we have
$$c^\pm (a^\pm\ast^\pm b^\pm,\til{\phi}\til{\psi} )\leq c^\pm (a^\pm, \til{\phi}) + c^\pm (b^\pm, \til{\psi}).$$

{\bf (Poincar\'e-Lefschetz duality)}\ Let $\Pi_2: QH_*^+\times QH^-_{2n-*}\to\FF$
be the non-degenerate pairing defined in ~\eqref{equation: Pi pairing on QH}. Then
$$
 c^\pm (a^\pm, H^\pm) =-\inf\left\{ c^\mp\(b^\mp, \(H^\pm\)^{(-1)}\)\ \vline\ \Pi_2(a^\pm, b^\mp)\neq 0\right\}.
$$
In particular, for any $\til{\phi}\in\tHam_c (M, \omega)$ we have
$$
 c^\pm\(a^\pm, \til{\phi}\) =-\inf\left\{ c^\mp \(b^\mp,\(\til{\phi}\)^{-1}\)\ \vline\ \Pi_2(a^\pm, b^\mp)\neq 0\right\}.
$$
\end{prop}
\begin{proof}

{\bf (Spectrality)}\ For $H^\pm\in\CL{H}_{lin, reg}^\pm$ it follows from the general algebraic theory of M. Usher,  see \cite[Corollary $1.5$]{Usher1}.  For degenerate $H^\pm$ ( in the rational case) we use the approximation technique due to Y.-G. Oh, see \cite[Theorem $6.4$]{Oh Spec for general1}.

{\bf (Quantum homology shift property)}\ It follows directly from the equation
$$\ell_{H^\pm}\(\Phi_{PSS}^\pm(\lambda a^\pm)\) = \ell_{H^\pm}\(\lambda\Phi_{PSS}^\pm( a^\pm)\) = \ell_{H^\pm}\(\Phi_{PSS}^\pm(a^\pm)\)+\nu(\lambda).$$

{\bf (Monotonicity)}\ It follows from Proposition~\ref{prop: well definition of c^pm}, item $(iii)$.

{\bf ($ C^0$-continuity)}\ It follows from Corollary~\ref{cor:extension of selectors to compact support}.

{\bf (Symplectic invariance)}\ It follows verbatim from  \cite[Theorem $5.9$, item $3$]{Oh Spec for general1}.

{\bf (Normalization)}\ It follows from \cite[formula $(5.14)$]{Oh1} .

{\bf (Homotopy invariance)}\ Given a Hamiltonian $H_t$, denote $\cl{H_t}:=H_t-\int_MH_t\omega^n$. Let $H^\pm, K^\pm\in\CL{H}_{lin, reg}^\pm$, such that $H^\pm\sim K^\pm$. Then $\cl{H^\pm}\sim \cl{K^\pm}$ and using the argument in  \cite[Theorem $I$]{Oh Normalization}, we get that $\Spec(\cl{H^\pm}) = \Spec(\cl{K^\pm})$. Repeating the argument  of  \cite[Theorem $6.1$]{Oh Spec for general1}, we conclude that
\bean
&c^\pm(a^\pm,\cl{H^\pm})=c^\pm(a^\pm,H^\pm)-\int_0^1\int_MH^\pm
\omega^ndt=c^\pm (a^\pm, \cl{K^\pm})\\&=c^\pm(a^\pm,K^\pm)-\int_0^1\int_MK^\pm\omega^ndt.
\eea

If $H, K\in\CL{H}_c(M)$ and $H\sim K$ then  $\cl{H}\sim\cl{K}$.
Following the proof of  \cite[Theorem $6.1$]{Oh Spec for general1}, we take sequences $\{H_n^\pm\}_{n\in\nn}, \{K_n^\pm\}_{n\in\nn}\subset
\CL{H}_{lin, reg}^\pm$, which $C^{\infty}$-converge to $H$ and $K$ respectively. Then $K\# H_n^\pm\#K^{inv}\in\CL{H}_{lin, reg}^\pm$, where $K^{inv}(t,x):=-K(t,\phi^t_K(x))$ generates $(\phi^t_K)^{-1}$,  and by the symplectic invariance we have
$$
c^\pm(a^\pm,H_n^\pm)=c^\pm(a^\pm,K\# H_n^\pm\#K^{inv}).
$$
On the other hand, since $H\sim K$,
$$
K\# H_n^\pm\#K^{inv}\sim K\# H_n^\pm\#H^{inv}.
$$
Thus
\bean
&c^\pm(a^\pm, K\# H_n^\pm\#K^{inv})-
\int_0^1\int_M(K\# H_n^\pm\#K^{inv})\omega^ndt\\
&=c^\pm(a^\pm,K\# H_n^\pm\#H^{inv})-
\int_0^1\int_M(K\# H_n^\pm\#H^{inv})\omega^ndt.
\eea
By taking the limit $n\to +\infty$, using the continuity of the spectral numbers and the fact that $H\#H^{inv}=0$, we get
$$c^\pm(a^\pm, H)-\int_0^1\int_M(K\# H\#K^{inv})\omega^ndt=c^\pm(a^\pm, K)-\int_0^1\int_MK\omega^ndt.
$$
But, $\int_0^1\int_MH\omega^ndt=\Cal(H)$ is the Calabi invariant and
$$
K\# H\#K^{inv}\sim K\# H\#H^{inv}=K
$$
implies that $\Cal(K\# H\#K^{inv})=\Cal(K)$.

{\bf (Triangle inequality)}\  Let $H, K\in\CL{H}_c(M)$. Take $\widetilde{H}^\pm, \widetilde{K}^\pm$ and $\widetilde{H\#K}^\pm\in\CL{H}_{lin, reg}^\pm$ approximating $H, K$ and $H\#K$ respectively. Moreover, if $A_{\widetilde{H}^\pm}, A_{\widetilde{K}^\pm}$ and $A_{\widetilde{H\#K}^\pm}$ are the slopes of the corresponding completions, we assume that
$$
2A_{\widetilde{H}^\pm} =2 A_{\widetilde{K}^\pm} = A_{\widetilde{H\#K}^\pm}=:A^\pm.
$$

Now take $G^\pm\in\CL{H}_{lin, reg}^\pm$ with slope $A^\pm$ of its completion. Then the continuation maps for pairs $\(\widetilde{H}^\pm, \tfrac12G^\pm\)$, $\(\widetilde{K}^\pm, \tfrac12G^\pm\)$ and $\( G^\pm, \widetilde{H\#K}^\pm\)$ are isomorphisms (actually identities),  see  \cite[Lemma $11$]{Rit3}. In particular, the pair-of-pants products $\ast_{PP}^\pm$ are also defined for a triple $\(\widetilde{H}^\pm, \widetilde{K}^\pm, \widetilde{H\#K}^\pm\)$. Moreover, by taking monotone-decreasing homotopies in the continuation construction for the pairs $\(\widetilde{H}^\pm, \tfrac12G^\pm\)$, $\(\widetilde{K}^\pm, \tfrac12G^\pm\)$ and $\( G^\pm, \widetilde{H\#K}^\pm\)$, see \cite[Section $2.9$]{Rit1},  and by gluing the corresponding continuation cylinders with a pair-of-pants surface for a triple $\(\tfrac12G^\pm, \tfrac12G^\pm, G^\pm\)$ , we get a sharp energy estimate for the triple $\(\widetilde{H}^\pm, \widetilde{K}^\pm, \widetilde{H\#K}^\pm\)$ by summing up energy estimates coming from the corresponding continuation cylinders, see \cite[Section $2.9$]{Rit1}, and from  the sharp energy estimate for the pair of pants \eqref{equation:sharp energy estimate}.\\
Now, we can repeat the argument from \cite[Section $7.4$]{Oh Spec for general1}.

{\bf (Poincar\'e-Lefschetz duality)}\ It follows from the general algebraic theory of M. Usher, see \cite[Corollary $1.4$]{Usher2}.
\end{proof}

\newpage
\subsection*{Acknowledgement.}
I am beholden to  Michael Entov, who introduced me to this subject, guided and helped me a lot. I am grateful to Michael Polyak for his valuable suggestions and comments in the course of my work on this paper. I would like to thank  Alexander F. Ritter for paying my attention to his results on Floer cohomology.  This work was carried out at Max-Planck-Institut f\"{u}r Mathematik, Bonn, and I would like to acknowledge its excellent research atmosphere and hospitality. Finally, I would like
to thank an anonymous referee for many valuable suggestions, comments and corrections.

\end{document}